\documentclass[article]{amsart}

\usepackage[utf8]{inputenc}
\usepackage{amsmath, amsthm, amssymb, amsfonts}

\newcommand\MA{\text{MA}}
\newcommand\Ric{\text{Ric}}
\usepackage[all]{xy}

\newtheorem{theorem}{Theorem}[section]
\newtheorem{lemma}[theorem]{Lemma}
\newtheorem{proposition}[theorem]{Proposition}
\newtheorem{corollary}[theorem]{Corollary}
\newtheorem{definition}[theorem]{Definition}

\newcommand{\C}{{\mathcal C}}
\newcommand{\LLL}{{\mathcal L}}

\newcommand{\maths}[1]{{\mathbb #1}}  
\newcommand{\RR}{\maths{R}}
\newcommand{\NN}{\maths{N}}
\newcommand{\CC}{\maths{C}}

\newcommand{\ZZ}{\maths{Z}}

\newcommand{\bbb}{{\mathfrak b}}
\newcommand{\aaa}{{\mathfrak a}}
\renewcommand{\ggg}{{\mathfrak g}}

\newcommand{\ppp}{{\mathfrak p}}
\newcommand{\hhh}{{\mathfrak h}}

\newcommand{\zzz}{{\mathfrak z}}
\newcommand{\CCC}{{\mathfrak C}}

\newcommand{\ttt}{{\mathfrak t}}

\newcommand{\kkk}{{\mathfrak k}}
\newcommand{\NNN}{{\mathfrak N}}
\newcommand{\XXX}{{\mathfrak X}}
\newcommand{\qqq}{{\mathfrak q}}

\title{Kähler-Ricci solitons on horospherical manifolds}
\author{F. Delgove}
\begin{document}

\begin{abstract}
    In this paper we prove there exists a Kähler–Ricci soliton on any smooth Fano horospherical manifold by restricting the Monge-Ampère equation of Kähler-Ricci soliton to a real Monge-Ampère equation and by using the continuity method. Finally, we compute the lower Ricci Bound in this case. We give a geometric interpretation related to the barycenter of the polytope moment associated to the horospherical manifold.
\end{abstract}

\maketitle
\section{Introduction}

The founding paper on the Kähler-Ricci solitons is Hamilton's article \cite {Hamilton}. They are natural generalizations of Kähler-Einstein metrics and appear as fixed points of the Kähler-Ricci flow. On a Fano compact Kähler manifold $M$, a Kähler metric $g$ is a \textit{Kähler-Ricci soliton} if its Kähler form $\omega_g$ satisfies :
$$
\operatorname{Ric}(\omega_g) - \omega_g = \LLL_X \omega_g,
$$
where $Ric(\omega_g)$ is the Ricci form of $g$ and $\LLL_X \omega_g$ is the Lie derivative of $\omega_g$ along a holomorphic vector field $X$ on $M$. Usually, we denote the Kähler-Ricci soliton by the pair $(g,X)$ and $X$ is called the \textit{solitonic vector field}. We immediately note that if $X=0$ then $g$ is a Kähler-Einstein metric. When $X \neq 0$, we say that the Kähler-Ricci soliton is \textit{non-trivial}. 

The first study of the solitonic vector field $X$ was done in the paper \cite{TZ1,TZ2}. Thanks to the Futaki function, the authors discovered an obstruction to the existence of Kähler-Ricci soliton and proved that $X$ is in the center of a reductive Lie subalgebra $\eta_r(M)$ of the space $\eta(M)$ of all holomorphic vector fields. This study also gives us a uniqueness result about Kähler-Ricci soliton (theorem 0.1 in \cite{TZ1}).

Subsequently, the study was developped by Wang and Zhu in \cite{WZ} where they show the existence of Kähler-Ricci solitons on toric manifolds using the continuity method. This work was supplemented by a study of the Ricci flow by Zhu in \cite{ZZZ} on the toric manifold which showed that the Kähler-Ricci flow converges to the Kähler-Ricci soliton of the toric variety. The result about existence of Kähler-Ricci solitons has been extended to cases of toric fibrations by Podesta and Spiro in \cite {MR2658183}. Recently, the result concerning the convergence of the Ricci flow has been also extended in \cite{2017arXiv170507735H}. 

In 2015, Delcroix used the approach of Zhu and Wang in the case of Kähler-Einstein metrics on some compactifications of reductive groups. In his paper \cite{Delcroix}, the main result is a necessary and sufficient condition for the existence of a Kähler-Einstein metric in some group compactifications. The condition is that the barycenter of the polytope associated to the group compactification must lie in a particular subset of the polytope. The first tool used in his proof is a study of the $(K \times K$)-invariant functions (for the $KAK$ decomposition), in particular he computes the complex Hessian of a $(K \times K)$-invariant function. The second tool is an estimate of the convex potential associated to a $K \times K$-invariant metric on ample line bundles. Then he proves the main result by reducing the problem to a real Monge-Ampère equation and by obtaining $\C^0$ estimates along the continuity method. In our paper,  we extend this approach to smooth horospherical manifolds in the following way.

\begin{theorem}\label{aaaa}
Assume that $M$ is a Fano horospherical manifold. There exits a Kähler-Ricci soliton $(X,g)$ on $M$.
\end{theorem} 

This result was already proved in \cite{Delcroix2} in a more general case. But in our paper, we focus on the case of smooth horospherical varieties and give a direct proof in this case. In order to prove this, we don't use the K-stability (as in \cite{Delcroix2}) and we prefer to use analytic methods such as the continuity method. As first step, we compute the Futaki invariant and use the results of \cite{TZ2} in order to get the expression of the solitonic vector field. As a second step, we compute the Monge-Ampère solitonic equation in the horospherical case and use the continuity method to conclude as in the toric case following the approach of \cite{WZ}. 

An important corollary that comes directly from the article \cite{Pasquier2009} is that there exist horospherical varieties which admit a non-trivial Kähler-Ricci soliton and therefore do not admit Kähler-Einstein metrics. Indeed, Matsushima theorem says that if the Fano variety has a non reductive group of automormphisms then it does not admit Kähler-Einstein metrics. In \cite{Pasquier2009}, Pasquier shows that there exists infinitely many horospherical varieties whose group of automorphisms is non reductive, so by using the previous theorem, the only possibility is that the soliton must be non trivial. This is summarized in the following corollary.

\begin{corollary}
There exists infinetely many horospherical manifolds admitting a non trivial Kähler-Ricci soliton.
\end{corollary}

Moreover, following the approach of \cite{Delcroix}, we compute the greatest Ricci lower bound $R(M)$ in the horospherical case. The latter was introduced in \cite{sze} in order to measure the default to " being of Kähler-Einstein" for a Fano manifold. We give a geometric interpretation related to the moment polytope associated with the horospheric variety. More precisely, we will show the following result.

\begin{theorem}
Assume that $ M $ is a smooth horospherical manifold with associated horospherical homogeneous space $ G / H $ such that $ H $ contains the opposite Borel subgroup  $ B^- $ of $ G $. Let us denote by $ P = N_G (H) $ and by $ \Delta ^ + $ the moment polytope with respect to $ B $ of $ M $. Moreover, assume that $ 2 \rho_P \neq \operatorname{Bar}_{DH} (\Delta^+) $ where $ \operatorname{Bar}_{DH} (\Delta^+) $ is the barycenter of the polytope $ \Delta^+ $ for the Duistermaat-Heckman measure. Then $R(M)$ is the unique $ t \in \, ] 0,1 [$ such that
$$
\dfrac{t}{t-1}(\operatorname{Bar}_{DH}(\Delta^+) + 2 \, \rho_P) \in \partial \left( \Delta^+  + 2 \, \rho_P \right).
$$
\end{theorem}

\paragraph{ \textit{Acknowledgements}} The author would like to thank T. Delcroix for helpful discusions about horospherical varieties and F. Paulin for his help for the redaction of this paper. This paper is extracted from the PhD thesis of the author realized under the supervision of N. Pali. 

\section{Horospherical Varieties}

In this section, we give some reminders on the theory of algebraic groups and on horospherical varieties. A good reference for the group theory part is \cite{Spr98}. For the notion of horospherical variety, we refer to \cite{Pasquier2009,Tim11}.

\subsection{Reductive group}

Let $G$ be a  reductive connected linear complex algebraic group. We denote $\ggg$ its Lie algebra. If $K$ is a maximum compact subgroup of $G$ with Lie algebra $\kkk$ then 
$$ \ggg = \kkk \oplus J \kkk, $$
where $ J $ is the complex structure of $ \ggg $. Fix a maximum torus of $ T $ of a Borel subgroup $ B $ of $G$. Denote by $ \Phi \subset \XXX (T) $ the root system of $ (G, T) $ where $ \XXX (T) $ is the group of algebraic characters of $ T $. We have the root decomposition:
$$
\ggg = \ttt \oplus \bigoplus_{\alpha \in \Phi} \ggg_\alpha$$
where for any $\alpha \in \XXX(T)$, $ \ggg_\alpha: = \lbrace x \in \ggg ~:~ \forall h \in \ttt ~ ~ \operatorname{ad}(h) (x) = \alpha (h) x, \rbrace$, so that $\ggg_\alpha$ is a complex line if and only if $\alpha \in \Phi$.

Let $ \Phi^+ :=\Phi^+(B)$ be the set of \textit{positive roots} (associated with $ B $) so that the Lie algebra $\bbb$ of $B$ satisfies
$$
\bbb = \ttt \oplus \bigoplus_{\alpha \in \Phi^+} \ggg_\alpha.
$$
We then define the negative roots $\Phi^-:=\Phi^-(B)=-\Phi^+$ of $\Phi$ associated with $B$ so that
$
\Phi = \Phi^+ \sqcup \Phi^-.
$
Let  $B^-$ be the unique Borel subgroup of $G$ called the \textit{opposite Borel subgroup} of $B$ with respect to $ T $ verifying $B \cap B^- =T$. Note that $\Phi^+(B^-)=\Phi^-(B)$.

Denote by $\Sigma$ the set of \textit{simple roots} as the set of roots in $\Phi^+$ that cannot be written as the sum of two elements of $\Phi^+$. For any subset $I$ of $\Sigma$, if we give a subset $I$, let $\Phi_I$ be the subset of $\Phi$ generated by the roots contained in $I$. \textit{The parabolic group containing $B$} with respect to $ I$, denoted by $P:=P_I$ is the connected closed subgroup of $G$ whose Lie algebra is
$$
\ppp = \ttt \oplus \bigoplus_{\alpha \in \Phi^+ \cup \Phi_I} \ggg_\alpha.
$$
Let
$
\Phi_P := \Phi^+ \cup \Phi_I.
$
The parabolic subgroup opposed to $P$, denoted by $Q:=Q_I$, is the parabolic subgroup associated with $I$ for the Borel subgroup $B^-$ i.e. the connected closed Lie subgroup with Lie algebra
$$
\qqq = \ttt \oplus \bigoplus_{\alpha \in \Phi^- \cup \Phi_I} \ggg_\alpha.
$$
Moreover $L = P \cap Q$ is a Levi subgroup of $P$. Let $\Phi^+_P$ be the set of roots that are not in $\Phi_Q = \Phi^- \cup \Phi_I$ so that 
$
\Phi^+_P \sqcup \Phi_Q = \Phi.
$
Note that $\Phi^+_P$ is the set of roots of the unipotent radical $U$ of $P$, that
$
\Phi = \Phi^+_P \sqcup \Phi_I \sqcup \Phi_Q^+,
$
and that
$
\Phi^+_Q=-\Phi^+_P.
$
Finally, we define
$$
2\rho_P := \sum_{\alpha \in \Phi^+_P} \alpha.
$$

\subsection{Horospherical subgroups and homogeneous horospherical spaces}

 If $H$ a closed connected algebraic subgroup of $G$, then $H$ is said to be a \textit{horospherical subgroup} of $G$ if $H$ contains the unipotent radical $U$ of a Borel subgroup  $B$. We can build horospherical subgroups using parabolic subgroups. Let $N_G(H):=\lbrace g \in G ~:~ gHg^{-1}=H \rbrace$ be the normalizer of $H$ in $G$.

\begin{proposition}[\cite{pasquierthese}]\label{PHt}
Let $H$ be a horospherical subgroup of $G$. Then $P:=N_G(H)$ is a parabolic subgroup of $G$ containing the Borel subgroup $B$, and the quotient $P/H=T/T\cap H$ is a torus.
Conversely, if $H$ is a connected closed algebraic subgroup of $G$ such that $N_G(H)$ is a parabolic subgroup $P$ of $G$ and $P/H$ is a torus, then $H$ is horospherical. The fibration
$$
G/H \longrightarrow G/P.
$$
is a torus fibration over a generalized flags manifold with fiber egal to $P/H$.
\end{proposition}

We obtain the following decomposition of the Lie algebra $\hhh$ of $H$:
$$
\hhh = \ttt_0 \oplus \bigoplus_{\alpha \in \Phi_P} \ggg_\alpha,\text{ where } \ttt_0 = \ttt \cap \hhh.
$$

If $H$ is a horospherical subgroup of $G$, then $G/H$ is called a \textit{homogeneous horospherical space}. On $G/H$, the normalizer $P:=N_G(H)$ acts by right multiplication by the inverse and $H$ is included in $P$ and acts trivially. We define an action of $G \times P/H$ on $G/H$ by 
$$
(g,pH) \cdot xH = gxp^{-1}H.
$$
Note that the isotropy group of $eH$ is $\lbrace (p,pH) \in G \times P/H ~:~ p \in P \rbrace$. This group is called $diag(P)$ and is isomorphic to $P$ by the first projection.

\subsection{Character groups and one parameter subgroup}
We denote 
$ 
\aaa = \ttt \cap J \kkk. 
$
We have an identification between $ \aaa $ and $ \NNN (T) \otimes_\ZZ \RR $, where $ \NNN (T) $ is the group of algebraic one-parameter subgroups $\lambda : \CC^\times \rightarrow T$ of $T$ given by the derivative at point $1$ of the restriction of $\lambda$ to $\RR^+_*$. Since, as $T \cap H$ is a subtorus of $T$ then $\NNN(T \cap H)$ defines a sublattice of $\NNN(T)$ corresponding to the one parameter subgroups have values in $T \cap H$. With $\aaa_0 = \NNN(T \cap H) \otimes_\ZZ \RR$, we have
$
\ttt_0 = \aaa_0 \oplus J \aaa_0.
$

Recall that the Killing form $\kappa$ of $\ggg$ defines a scalar product $ ( \cdot, \cdot) $ on $\aaa \cap[\ggg, \ggg]$. In addition, $\kappa$ is egal to zero on $\zzz(\ggg)$. Thus we can define a global scalar product $\aaa$ by taking a scalar product invariant by the Weyl group $W$ on $ \aaa \cap Z(\ggg) $ and assuming thaht $ \aaa \cap \zzz(\ggg) $ and $ \aaa \cap [\ggg, \ggg] $ are orthogonal. Let $\aaa_1$ be the orthogonal of $\aaa_0$ for the scalar product $ (\cdot, \cdot) $ and so that
$
\ttt = \aaa_0 \oplus \aaa_1 \oplus J \aaa_0 \oplus J \aaa_1
$
and 
$
\ppp/\hhh \simeq \aaa_1 \oplus J \aaa_1
$.

Finally, we recall that there is a natural pairing $ \langle \cdot, \cdot
\rangle $ between $ \NNN (T) $ and $ \XXX (T) $ defined by $ \chi \circ \lambda (z) = z^{\langle \lambda, \chi \rangle} $ for all $\lambda \in \NNN(T), \chi \in \XXX(T)$ and $z \in \CC^\times$. In addition, the natural pairing between $ \Lambda= \XXX (T) \otimes_\ZZ \RR $ and $ \NNN (T) \otimes_\ZZ \RR $ obtained by  $\RR$-linearity can be seen as $ \langle \chi, a \rangle = \ln \chi (\exp a) $ for all $ \chi \in \XXX(T) $ and $ a \in \aaa \simeq \NNN (T) \otimes_\ZZ \RR $. Since $(\cdot,\cdot)$ is a scalar product, for $ \chi \in \XXX (T) $, we denote $ t_\chi $ the unique element of $ \aaa $ such that 
\begin{equation}{\label{Dada2}}
\forall a \in \aaa, ~~ (t_\chi, a) = \langle \chi, a \rangle.
\end{equation}
For every $\alpha \in \Phi$, let $e_\alpha$ the generator of the complex line $\ggg_\alpha$ such that $[e_\alpha, e_{-\alpha}] = t_\alpha$. Let's end this section by recalling the polar decomposition.
\begin{proposition}[\cite{Delcroix2}]\label{polar}
The image of $\aaa_1$ in $G$ under the exponential is a fundemental domain for the action of $K \times H$ on $G$, where $K$ acts by multiplication on the left and $H$ by multiplication on the right by the inverse. As a consequence, the set $\lbrace \exp(a)H ~:~ a \in \aaa_1 \rbrace$ is a fundamental domain for the action of $K$ on $G/H$.
\end{proposition}

\subsection{Horospherical variety}

Recall that a complex algebraic  $G$-variety is a reduced finite type scheme over $\CC$ with an algebraic action of $G$. A normal complex algebraic $G$-variety $X$ will be said to be \textit{$G$-spherical} if it admits an open and dense $B$-orbit. Note that $X$ is then connected. A $G$-spherical variety $X$ will be called \textit{horospherical} if the stabilizer $H$ in $G$ of a point $x$ in the open and dense $B$-orbit is horospherical. We then will say that $(X,x)$ is \textit{a horospherical embedding}. In particular, $X$ has $G/H$ as a dense open set. Two horospherical embeddings $(X,x)$ and $(X',x')$ are \textit{isomorphic } if there is an $G$-equivariant isomorphism from $X$ to $X'$ sending $x$ to $x'$. In this paper, we will assume that $X$ is smooth. Thanks to the GAGA theorems (see \cite{AIF_1956__6__1_0}), the variety $X$ is therefore a Fano projective complex manifold and in particular $X$ is a connected compact Fano Kähler manifold for the metric induced by the Fubini-Study metric of the projective space.

We fix a homogeneous horospherical space $G/H$. Recall that there is a unique element $\alpha^\vee \in \XXX(T)^* \simeq \NNN(T)$ such as $\langle \alpha, \alpha^\vee \rangle=2$. With our previous notations, we have 
$
\alpha^{\vee}= \frac{2}{\Vert t_\alpha \Vert^2} \, t_\alpha.
$
Define $a_\alpha \in \ZZ$ by $ a_\alpha=\langle 2\rho_P, \alpha^\vee \rangle$ and \textit{ Weyl's dominant closed chamber} $\CCC$ by
$$
\CCC = \lbrace \chi \in \Lambda ~:~ \langle \chi, \alpha^\vee \rangle \geq 0 ~~ \forall \alpha \in \Phi^+ \rbrace.
$$
and \textit{the open dominant Weyl chamber} as the interior of the closed dominant Weyl chamber. The semigroup of  \textit{dominant weights} is defined by $\Lambda^+:= \XXX(T) \cap \CCC$.

We have the following definition introduced in \cite{pasquierthese}.
\begin{definition}\label{GHref}
Let $G/H$ be a homogeneous horospherical space. A convex polytope $Q$ of $\aaa \simeq \NNN(T) \otimes_\ZZ \RR$ is said to be \rm{$G/H$-reflective} if we have the following three conditions:
\begin{enumerate}
\item[(1)] $Q$ has its vertices in $\NNN(T) \cup \left\lbrace \frac{\alpha^\vee}{a_\alpha} ~:~ \alpha \in \Phi^+_P \right\rbrace$ and contains $0$ in its interior,
 
\item[(2)] the dual polytope $Q^*$ has its vertices in $\XXX(T)$,
\item[(3)] for all $ \alpha \in \Phi^+_P$, we have $\cfrac{\alpha^\vee}{a_\alpha} \in Q$.
\end{enumerate}
\end{definition}

The \textit{moment polytope} $\Delta^+ \subset \Lambda$ with respect to the Borel subgroup $B$ of the horospherical manifold $X$ as the \textit{Kirwan's moment polytope} of the Kähler manifold $(X,\omega)$ for the action of a maximum compact subgroup $K$ of $G$, where $\omega$ is a $K$-invariant Kähler form in $ 2 \pi \, c_1(X)$ (see \cite{Bri87,K1} for more details). Another way is through representation theory. Let $L$ be an ample $G$-linearized line bundle of a $G$-spherical variety $X$. Denote by $V_\lambda$ an irreducible representation of $G$ of higher weight $\lambda \in \XXX(T)$ with respect to the Borel subgroup $B$. Since $X$ is $G$-spherical, for every $r \in \NN$, there exits a finite set $\Delta_r \subset \XXX(T)$ such that $H^0(X,L^r) = \oplus_{\lambda \in \Delta_r} V_\lambda$. The moment polytope $\Delta^+$ with respect to $B$ is the closure of $\cup_{r \in \NN} \frac{\Delta_r}{r}$ in $\XXX(T) \otimes_\ZZ \RR$ ( \cite{Bri87, Bri89} for more details). We then have the following result.

\begin{proposition}[\cite{pasquierthese}]\label{roro1}
Let $G/H$ be a homogeneous horospherical space. There exists a bijection between Fano horospherical  embedding of $G/H$ and the set of $G/H$-reflective polytopes in $\aaa$. 
In addition, the polytope $ 2 \rho_P + Q^*$ is the moment polytope with respect the Borel subgroup $B$ of the horospherical embedding. The assumption $(3)$ of the definition \ref{GHref} is then equivalent to the fact that $ \Delta^+=2 \rho_P + Q^*$ is included in $\CCC$. In particular, we see that $0 \in \operatorname{Int}(Q^*)$ and that therefore $2 \rho_P \in \operatorname{Int}(\Delta^+)$.
\end{proposition}

We fix a horospherical embedding $(X,x)$ of $G/H$ and we denote $P:=N_G(H)$. By taking the restriction to the open $B$-orbit, we have an isomorphism between the $G$-equivariants automorphisms of $X$ and those of $G/H$ :
\begin{equation}\label{Aut_G(X)}
\operatorname{Aut}_G(X) \simeq \operatorname{Aut}_G(G/H) \simeq P/H.
\end{equation}
We can consult \cite{Knop91} which deals with the problem in the more general context of spherical varieties.

\subsection{Associated linearized line bundle}
In this section, we introduce associated linearize line bundle over homogeneous horospherical space. They are introducted firstly by Delcroix in \cite{Delcroix2}. Let $H$ be horospherical subgroup of $G$ and $P=N_G(H)$.

\begin{definition}
A line bundle $\Pi: L\rightarrow G/H$ is $(G \times P/H)$-linearized if there exists an action $\Theta : ( G \times P/H) \times L \rightarrow L$ noted $(x,y) \mapsto \Theta_x(y)$ such that
\begin{itemize}
\item[$\bullet$] $\Pi : L \rightarrow M $ is a $(G \times P/H)$-equivariant morphism,
\item[$\bullet$] the application induced by the action of $\Theta_{{(e,pH)}} $ between the fibers is linear.
\end{itemize}
\end{definition}
Note that to any $(G \times P/H)$-linearized line bundle, we can associate a character of $P$. Indeed, for any $p \in P$, the action of $(p,pH) \in diag(P)$ is trivial on the trivial class in $G/H$. Thus, $\Theta_{(p,pH)}$ induces a linear isomorphism  between $L_{eH}$ and $L_{eH}$ and therefore a linear representation of dimension 1 of $diag(P) \simeq P$ i.e. there is a character $\chi \in \XXX(P)$ such that
\begin{equation}\label{chixi}
(p,ph) \cdot \xi = \chi(p) \xi,~~ \forall \xi \in L_{eH}.
\end{equation}
Let us consider the projection $\pi :G \longrightarrow G/H$. We can define the pulled-back line bundle $\pi^*L$ over $G$. In addition, since $\pi$ is $G$-equivariant, the line bundle $\pi^*L$ admits a $G$-linearization. In particular, we define a global section $s$ on $\pi^*L$ by chosing an element $s(e) \in (\pi^*L)_e$ and setting
\begin{equation}\label{sdef}
s: g \in G \mapsto g \cdot s(e) \in (\pi^*L)_g.
\end{equation}
Now let us consider the inclusion $\iota : P/H \longrightarrow G/H$. We can define the restriction $\iota^*L$ of the line bundle $L$. In addition, since $\iota$ is equivariant for the action of $P \times P/H$, we obtain that $\iota^*L$ is a $(P \times P/H)$-linearized bundle. Two global sections of $\iota^*L$ can also be defined:
\begin{align}\label{sr}
& s_r : pH \in P/H \mapsto (p,eH) \cdot s'(e) ,\\
& s_l : pH \in P/H \mapsto (e,p^{-1}H) \cdot s'(e),
\end{align}
where the element $s'(e) \in (\iota^*L)_e$ is such that $s'(e)$ and $s(e)$ are mapped to the same  element of $L_{eH}$ by the canonical applications $\pi^*L \rightarrow L$ and $\iota^*L \rightarrow L$. Note that these sections are linked by the formula:
$$
\forall p\in P,~~s_r(pH)= \chi(p) s_l(pH).
$$
We have the following commutative diagram
$$
 \xymatrix{
    \pi^*L \ar[r] \ar[d] & L \ar[d] & \iota^*L \ar[l] \ar[d] \\
    G \ar@/^1pc/[u]^s \ar[r]_\pi & G/H & P/H \ar@/^{-1pc}/[u]_{s_r} \ar@/^1pc/[u]^{s_l} \ar[l]^{\iota} 
  }
$$
Given a Hermitian metric $q$ on a complex line bundle $L$ over a complex manifold $M$ and a local trivialisation $s$ of $L$ over an open subset $U$ of $X$, we define the local potential of $q$ with respect to $s$ by $\varphi : x \in U \mapsto - \ln a_x$ where $a_x= \vert s(x) \vert_q^2$. We can associate, to any Hermitian metric $q$, a $(1,1)$-form $ \omega_q $ called the \textit{curvature of $ q $} by $ \omega_q\vert_U = \sqrt{-1} \partial \overline {\partial} \varphi $ where $ \varphi $ is the local potential. One checks that $ \omega_q \vert_U $ does not depend on the trivialisation and therefore defines a $ (1,1) $-global form. In addition, one can prove that $ \omega_q \in 2  \pi \, c_1 (L) $. We will also say that \textit{$L $ has positive curvature} if there is a metric $ q $ such that $ \omega_q $ is a Kähler form. Fix a reference Hermitian metric $ q ^ 0 $ on $L$ and for any Hermitian metric $ q $, we define a smooth function $ \psi $ on $ X $, called the \textit{global potential of $ q$ with respect to $ q^0 $} by 
$$
\forall x \in X,\forall \xi \in L_x,~~\vert \xi \vert_q^2 = e^{- \psi(x)} \vert \xi \vert^2_{q^0}.
$$
By definition of the $(1,1)$-forms $\omega_q$ and $\omega_{q^0}$ and by computing in local charts, we see that the function $\psi $ satisfies the following relation:
$$
\omega_{q^0} = \omega_q + \sqrt{-1} \, \partial \overline {\partial} \, \psi.
$$
We refer to \cite{Demailly1} for more details.

Let $ G / H $ be a homogeneous horospherical space, $ L $ a $ (G \times P / H )$-linearized line bundle over $ G / H $ and $ q $ a Hermitian $ K $- invariant metric  on $ L $. We can then consider the Hermitian metric $\pi^*q$ on $\pi^*L$ and define the local potential $\phi$ (which is actually defined on the whole $G$) with respect to the section $s$ of equation \eqref{sdef} :
$$
\phi : g \in G \mapsto -2 \ln \vert s(g) \vert_{\pi^*q} \in \RR.
$$
We also define the potential $ u: \aaa_1 \rightarrow \RR $ associated with $ \iota^*L $:
\begin{equation}\label{potudef}
u : x \in \aaa_1 \mapsto -2 \ln \vert s_r(\exp(x)H) \vert_{\iota^*q} \in \RR.
\end{equation}
We have the following relation between these two potentials, using the character $\chi$ defined in equation \eqref{chixi}.
\begin{proposition}[\cite{Delcroix2}]
For all $k \in K$, $x \in \aaa_1$ and $h \in H$, we have
$$
\phi(k \, \exp(x) \, h ) = u(x) - 2 \ln \vert \chi( \exp(x)h ) \vert.
$$
\end{proposition}   

\subsection{Curvature in horospherical case}

Let us now, we recall Delcroix's computation of the curvature of a Hermitian metric on a line bundle over $ G / H $ in an adapted basis. The first step is to define this basis. For this, we identify the tangent space of $G/H $ at
$eH $ with $\ggg/\hhh \simeq \oplus_{\alpha \in \Phi^+_P} \CC e_{-\alpha} \oplus \aaa_1 \oplus J \aaa_1$. We get a complex basis of the tangent space $ T_{eH} (G / H )$ as the concatenation of a real basis $ (l_i)_{1 \leq i \leq i \leq r} $ of $ \aaa_1 $ with $ (e_ {-\alpha})_{\alpha \in \Phi^+_P} $. On $ P / H $, we can define for $ \xi \in T_{eH} (G / H )$ the holomorphic vector field:
\begin{equation}\label{Rxi}
R \xi : pH \mapsto (H,p^{-1}H) \cdot \xi.
\end{equation}
We then have a complex basis of $ T^{1,0} (P / H) $ given $ (Rl_j - \sqrt {-1} J Rl_j) / $2 for $ 1 \leq j \leq r$ and $ (R e_ {-\alpha} - \sqrt{ -1} JR e _{-\alpha}) / $2 for all $\alpha \in \Phi^+_P$, and we note by $ (\gamma_j)_{ 1 \leq j \leq r} \cdot (\gamma_\alpha)_{\alpha \in \Phi^+_P} $ the dual basis. 
\begin{theorem}[\cite{Delcroix2}]
Let $\omega$ be the curvature $(1,1)$-form of $K$-invariant Hermitian metric $q$ on a $(G \times P/H)$-linearized line bundle $L$ over $G/H$, whose associated character is denoted by $\chi$. The form $\omega$ is determined by its restriction to $P/H$, given for any $x \in \aaa_1$ by
$$
\omega_{\exp(x)H}= \sum_{1 \leq j_1,j_2 \leq r} \cfrac{1}{4} \cfrac{\partial^2 u}{\partial l_{j_1} \partial l_{j_2}}(x) \sqrt{-1} \, \gamma_{j_1} \wedge \overline{\gamma}_{j_2} + \sum_{\alpha \in \Phi^+_P} \langle \alpha, \frac{1}{2}\nabla u(x) - t_\chi \rangle \sqrt{-1} \, \gamma_\alpha \wedge \overline{\gamma}_\alpha
$$
where $\nabla u(x) \in \aaa_1$ is the gradient of the function $u$ defined by equation \eqref{potudef} for the scalar product $( \cdot, \cdot)$.
In addition, with $\operatorname{MA}_\RR(u)=\det(\operatorname{Hess}(u))$,
$$
\omega^n_{\exp(x)H} = \dfrac{\MA_\RR(u)(x)}{4^r 2^{\operatorname{Card}(\Phi^+_P})} \prod_{\alpha \in \Phi^+_P} \langle  \alpha , \nabla u(x) + 4 t_{\chi}  \rangle \, \Omega,
$$
where $n$ is the dimension of $G/H$ and
$$
\Omega:= \Big( \bigwedge_{1 \leq j \leq r} \gamma_j \wedge\overline{\gamma}_j \Big)  \wedge   \Big( \bigwedge_{\alpha \in \Phi^+_P} \gamma_\alpha \wedge \overline{\gamma}_\alpha \Big).
$$
\end{theorem}
Moreover, by choosing $s(e)$ appropriately, with $s_r(e)$ defined in equation \eqref{sr}, we have 
\begin{equation}\label{omegansr}
\omega^n_{\exp(x)H} = \dfrac{\MA_\RR(u)(x)}{4^r 2^{\operatorname{Card}(\Phi^+_P})} \prod_{\alpha \in \Phi^+_P} \langle  \alpha , \nabla u(x) + 4 t_{\chi} \rangle s_r^{-1}  \wedge \overline{s_r}^{-1}.
\end{equation}
In order to explain the choice of $s(e)$, let
$$
S:= \Big( \bigwedge_{1 \leq j \leq r} \gamma_i \Big) \wedge \Big( \bigwedge_{\alpha \in \Phi^+_P} \gamma_\alpha \Big), 
$$
which therefore defines a section of the  line bundle $ \iota^*K_{G/H}$ where $K_{G/H}$ is the canonical line bundle of $G/H$. In particular, we have 
$
S \vert_{eH} \in (\iota^*K)_{eH},
$
and using the following isomorphisms
$$
(K_{P/H})_{eH} \simeq (\iota^*K_{G/H})_{eH} \simeq (K_{G/H})_{eH} \simeq (\pi^*K_{G/H})_{eH},
$$
we denote via these isomorphisms $s(e):=S \vert_{eH}$. We then obtain, by the definition of $\Omega$ and $s_r$ (equations \eqref{Rxi} and \eqref{sr}), that
$$
 \forall x \in \aaa_1,~~s_r(\exp(x)H)=S \vert_{\exp(x)H},
$$
and we can conclude. We will constantly use this choice afterwards.

In this section, we consider a horospherical manifold with reductive group $G$ and horospherical subgroup $H$. Note $P=N_G(H)$ the normalizer of $H$ in $G$. Recall that $P$ is a parabolic subgroup and that there is another parabolic subgroup $Q$ such that there is a Levi subgroup $L$ such that $P \cap Q= L$.From now, we consider that the parabolic subgroup $P$ contains the opposite Borel subgroup $B^-$ i.e. $\Phi^+_P= \Phi^- \backslash \Phi_L = - ( \Phi^+ \backslash \Phi_L) $. Recall that we introduced the moment polytope $ \Delta^+ $ of $X$ with respect to the Borel subgroup $B$. Now, we define 
\begin{equation}\label{polytopeDelta} 
\Delta: = -2 \rho_P - \Delta^+
\text{ where } \rho_P = \sum_{\alpha \in \Phi^+_P} \alpha
\end{equation}
and \textit{the support function} $v_{2 \Delta}$ by
\begin{equation}\label{fctsupport}
v_{2 \Delta}: x \in 2 \Delta \mapsto  \sup_{p \in    2 \Delta} (x,p) \in \RR 
\end{equation}
This function satisfies the following properties:
\begin{itemize}
\item[$\bullet$] $\forall x \in \aaa_1^*,~~\forall \alpha \in \RR^+_*,~~ v(\alpha \, x) = \alpha v(x)$
\item[$\bullet$] $\forall (x,y) \in (\aaa_1^*)^2,~~ v(x+y) \leq v(x) + v(y)$
\item[$\bullet$] $\forall x \in \aaa_1^*, ~~ v(x) \leq d \, \Vert x \Vert $ where $d= \sup_{p \in 2 \Delta} \Vert p \Vert$ and $\Vert \cdot \Vert$ the euclidian norm on $\aaa_1$ (for any basis of $\aaa_1$) .
\end{itemize}
In addition, if $x \in \aaa_1^*$ is such that 
$\Vert x  \Vert=1$, then $2 \Delta$ is contained in the half-space $\lbrace y \in \aaa_1^*~:~ (x,y) \leq v_{2 \Delta}(x) \rbrace$ and at least one point of $2 \Delta$ is in the border of this half-space i.e. in $\lbrace y \in \aaa_1^*~~:~ (x,y) = v_{2 \Delta}(x) \rbrace $.

\begin{proposition}[\cite{Delcroix2}]\label{potconv}
Let $ q $ be $K$-invariant Hermitian metric with positive curvature on the line bundle $ K_X^{-1} $ and let $ u: \aaa_1 \rightarrow \RR $ be the convex potential defined by equation \eqref{potudef}. Then $u $ is a smooth and strictly convex function such that the application $ d u: x \mapsto d_x u \in \aaa_1^*$ verifies $\text{im}(du)=2 \Delta $ and the function $ u-v_{2 \Delta} $ is bounded on $ \aaa_1 $. In particular, the polytope $ 2 \, \Delta$ is independent of the chosen metric $q$.
\end{proposition}
Since $2 \rho_{Q} = -2 \rho_{P}$ belongs to $\operatorname{Int}(\Delta^+)$ by proposition \ref{roro1}, we have
\begin{equation}\label{0inD}
0 \in \operatorname{Int}(\Delta).
\end{equation}
Recall that
$
\Delta^+ \subset \CCC,
$
where $\CCC$ is the Weyl chamber for the Borel subgroup $B$ (see proposition \ref{roro1}). Thus 
$$
\forall \alpha \in \Phi^+(B), \forall p \in \Delta^+, ~~ \langle p , \alpha^\vee \rangle \geq 0.
$$
The latter can still be written
$$
\forall \alpha \in \Phi^+(B), \forall p\in  \Delta^+, ~~ ( p, \alpha) \geq 0.
$$
Since $\Phi^+_P \subset \Phi^{+}(B^{-})=-\Phi^{+}(B)$, we then have
$$
\forall \alpha \in \Phi^+_P, \forall p \in -\Delta^+, ~~ ( p, \alpha) \geq 0.
$$
Hence by compactness of $\Delta^+$, there exists $f>0$ such that
\begin{equation}\label{f>0}
\forall \alpha \in \Phi^+_P, ~~\forall p \in -\Delta^+, ~~ 0 \leq ( \alpha, p ) \leq f.  
\end{equation}
We end this section with the notion of the \textit{barycenter of a polytope}. We consider a probability measure $\nu$ on the polytope $\Delta^+$. For any endomorphism $f \in \LLL(\aaa_1^*, \aaa_1^*) $, we then consider the map
$$
\int_{  \Delta^+} f d\nu : z \in \aaa_1  \mapsto \int_{ x \in  \Delta^+} \langle f(x) , z \rangle \, d \nu(x) \in \RR.
$$
It follows directly from the definition that $\int_{ \Delta^+} f d\nu \in \aaa_1^*$. Taking the identity as application $f$, we call
$\int_{ \Delta^+} x \, d\nu$ \textit{the barycenter} of the polytope $\Delta^+$ for the measure $ \nu$ and we note it $\operatorname{Bar}_{d \nu}(\Delta^+)$. If $\nu$ is finite and not zero measure on $\Delta^+$ then we define the barycenter for the measure $\frac{\nu}{V}$ where $V$ is the volume of $\Delta^+$ for the measure $\nu$.

For the following, we define the \textit{ Duistermaat-Heckman measure} $\mu_{DH}$ on the polytope moment $\Delta^+$ as the measure with the density
$
p \mapsto \prod_{\alpha \in \Phi_Q^+} (\alpha,p),
$
with respect to Lebesgue measure on $\Delta^+$. We will then denote $ \operatorname{Bar}_{DH}(\Delta^+)$ the barycenter of $\Delta^+$ for the Duistermaat-Heckman measure.
\section{Existence of Kähler-Ricci solitons in the horospherical case}\label{setuphoro}

Let us fox a connected Fano Kähler manifold $M$. Let us recall (see for instance \cite{dercalabi} for detail on real and complex vector fields) that the group of complex automorphisms 
$\operatorname{Aut}(M)$ of $M$ is a finite dimensional Lie group whose Lie algebra is the set of holomorphic real vector fields denoted $ \eta^\RR(M)$ (see theorem 1.1 in chapter $3$ of \cite{Kobayashi2}). If $K$ is a maximum compact subgroup of the connected identity component $\operatorname{Aut}^{\circ}(M)$ of the holomorphic automorphisms group $\operatorname{Aut}(M)$ of the manifold $M$, then we have, see \cite{Fujiki1978}, that
$$
\operatorname{Aut}^{\circ}(M) = \operatorname{Aut}_r(M) \ltimes R_u,
$$
where $\operatorname{Aut}_r(M)$ is a reductive subgroup of $\operatorname{Aut}^{\circ}(M)$ and the complexification of $K$ and $R_u$ the unipotent radical of $\operatorname{Aut}^{\circ}(M)$. In addition, if we note $\eta^\RR(M)$, $\eta_r^\RR(M)$, $\eta_u^\RR(M)$ and $\kkk$ the Lie algebras of $\operatorname{Aut}(M),\operatorname{Aut}_r(M),R_u$ and $K$ respectively, then we have
$$
\eta^\RR(M)= \eta_r^\RR(M) \oplus \eta_u^\RR(M).
$$
Recall that $\eta^\RR(M)$ is the set of holomorphic real vector fields i.e. $X \in \eta^\RR(M)$ if and only if it checks $\LLL_X \omega$=0. Moreover, if we denote $\eta(M)$ the Lie algebra of the complex holomorphic vector fields i.e. the holomorphic sections of the complex vector bundle $T^{1,0}_J M$ then there is an isomorphism between $\eta^\RR(M)$ and $\eta(M)$ given by $X \mapsto X^{1,0}$. If we denote $\eta_r(M)$ and $\eta_u(M)$ the image of $\eta_r^\RR(M)$ and $\eta_u^\RR(M)$ respectively by the previous isomorphism, we then obtain a decomposition 
$$
\eta(M)= \eta_r(M) \oplus \eta_u(M).
$$
Assume that $(M,x)$ is a horospherical embedding under the action of the reducive group $\operatorname{Aut}_r(M)=:G$ and that $x \in M$ is such that its isotropy group $H$ in $G$ is the horospherical subgroup in $G$ containing the unipotent radical of the opposite Borel subgroup $B^-$. We also denote $P:=N_G(H)$. Using the decompositions in the previous section, with noting $\ggg= \eta^{\RR}_r(M)$ the Lie algebra of $G$, we have
$$
\ggg = \aaa_1 \oplus \aaa_0 \oplus J \aaa_1 \oplus J \aaa_0 \oplus \bigoplus_{\alpha \in \Phi^+_P} \ggg_\alpha \oplus \bigoplus_{\alpha \in \Phi_I} \ggg_\alpha \oplus \bigoplus_{\alpha \in \Phi^-_P} \ggg_\alpha.
$$
In particular, the Lie algebra of the Lie group $\operatorname{Aut}_G(M)$ of the $G$-equivariant automorphisms of $M$ is identified with the Lie algebra of $P/H$ which corresponds to the factor $\aaa_1 \oplus J \aaa_1$ in the previous decomposition. Note that, by definition of $\operatorname{Aut}_G(M)$, this Lie algebra also identifies with the center $\zzz(\ggg)$ of the $\ggg$.

We fix a Riemanian metric $g^0$ of Kähler form $\omega_{g^0} \in 2 \pi \, c_1(M) $ on M. By the $\partial
\overline{\partial}$-lemma, there is a unique function $h$ in $\C^{\infty}(M,\RR)$ such that
\begin{equation}\label{Dada1}
\Ric(\omega_{g^0}) - \omega_{g^0} = \sqrt{-1} \partial
\overline{\partial} h, ~~ \int_M e^h \omega_{g^0}^n = \int_M \omega_{g^0}^n.
\end{equation}
In addition, if we fix a Hermitian metric $ m^0 $ on $K_M^{-1} $ such that $ \omega_{m^0} = \omega_{g^0} $, then we can define a volume $ dV $ on $M$ given in a local trivialisation $ s $ of $ K_M^{-1} $ by
$$
dV= \vert s \vert_{m^0} s^{-1} \wedge \overline{s^{-1}} = e^{-\varphi} s^{-1} \wedge \overline{s^{-1}},
$$ 
where $\varphi$ is the local potential with respect to the trivialization $s$. Up to a additive constant, $ h $ is the logarithm of the potential of $ dV $ compared to $ \omega^n_{g^0} $, so we renormalize to match them i.e. we get
\begin{equation}\label{renormal}
e^{h} \omega_{g^0}^n = dV.
\end{equation}
Indeed, by writing locally in a trivialization open $U$ the equality \eqref{Dada1}, we get that
$$
\partial \overline{\partial} \left( \ln \det(g_{i \overline{j}}) +h \right) = \partial \overline{\partial} \left( \ln \det(\varphi_{i \overline{j}}) +h \right) = \partial \overline{\partial} \varphi,
$$
where $\varphi$ is the local potential in this open set. This last equation can finally be written globally 
$$
\partial \overline{\partial} \left( \ln e^h \, \cfrac{\omega^n_{g^0}}{dV} \right)=0. 
$$
We conclude by using the maximum principle.

\subsection{Determination of the solitonic vector field}

The first step in oder to prove the existence of Kähler-Ricci soliton is to determine the solitonic vector field. To do this, we use the Futaki invariant. We have the following result (see proposition $2.1$ in \cite{TZ2}):
\begin{proposition}\label{Fnuletar}
There exists a unique complex holomorphic vector field $X \in \ggg^{1,0}$ with $\operatorname{Im}(X) \in \kkk$ such that the Futaki invariant of $X$, noted $F_X : \eta(M) \rightarrow \CC$, vanishes on $\ggg^{1,0}$. Moreover, $X$ is equal to zero or belongs to the center of $\ggg^{1,0}$, and we have
$$
\forall (u,v) \in \ggg^{1,0} \times \eta(M),~~F_X([u,v])=0.
$$
In particular $F_X(\cdot)$ is a Lie character on $\ggg^{1,0}$.
\end{proposition}
By applying this theorem to the horospherical case, we obtain the following result
\begin{proposition}\label{formX}
The vector field $ X \in \eta_r(M)=\ggg^{1,0}$ given by proposition \ref{Fnuletar} has the following form:
$$
X = \xi - \sqrt{-1} J \,\xi \text{ where } \xi \in \aaa_1.
$$
\end{proposition}

\begin{proof}
We are looking for $X \in \zzz(\ggg^{1,0})$ such that
$
\text{Im} X \in  \kkk.
$
But
$\zzz(\ggg^{1,0}) = ( \ppp/\hhh)^{1,0} = \left( \aaa_1 + J \aaa_1 \right)^{1,0}$ and therefore $ \kkk \cap ( \aaa_1 \oplus J \aaa_1) = J \aaa_1,$
so there is a $\xi \in \aaa_1$ such that
$$
X = \xi - \sqrt{-1} I \,\xi \text{ where } \xi \in \aaa_1.
$$
\end{proof}

Let us now compute the Futaki invariant in the horospherical case. We recall that we can write it in the form (see \cite{TZ2})
$$
\forall Y \in \eta(M),~~ F_X(Y)= \int_M \tilde{\theta}_Y e^{\tilde{\theta}_X}, \omega_g^n, 
$$
where $\omega_g$ is a Kähler form belong to $c_1(M)$ and the map $\tilde{\theta}_X \in \C^\infty(M,\CC)$ satisfies
\begin{equation}\label{condtildetheta}
i_X \omega_g = \sqrt{-1} \, \overline {\partial} \, \tilde{\theta}_X,
~~
\int_M \tilde{\theta}_X e^{h_g} \omega^n_g = 0.
\end{equation}
Such a function exists for any $X \in \eta(M)$ by Hodge theory. Using the Cartan formula, we have
$$
\LLL_X \omega =  \sqrt{-1} \, \partial  \overline {\partial} \, \tilde{\theta}_X.
$$
First, we use the $K$-invariance of the Kähler form $ \omega_g $ to notice the following fact:
$$
\LLL_X \omega_g = \LLL _ {\xi - \sqrt{-1} J \xi} \omega_g
= \LLL_\xi \omega_g - \sqrt{-1} \LLL_{J \xi} \omega_g
= \LLL_\xi \omega_g,
$$
since $ \omega_g $ is $ K $-invariant and $ J\xi \in \kkk $. So we just compute $ \LLL_\xi \omega_g $. We have the following propostion (Proposition 4.14 of \cite{Delcroix} with a different convention, group action acts on the left).
\begin{proposition}
Let $\zeta \in \aaa_1$. If we denote $Y= \zeta - \sqrt{-1} J \zeta \in \aaa_1^{1,0}$ then the function $\theta_Y$ is $K$-invariant and
\begin{equation}\label{omegantildehoro}
\forall x \in \aaa_1,~~\tilde{\theta}_{Y} (\exp(x)H) =  ( \nabla u(x), \zeta ),
\end{equation}
where $\nabla u$ is the gradient of the convex potential $u$ defined by equation \eqref{potudef} for the scalar product $( \cdot, \cdot)$. 
\end{proposition}
With this result, we can compute the Futaki invariant $F_X$ in order to prove the following result.
\begin{proposition}\label{Paul6}
Let $X=\xi - \sqrt{-1}J \xi$ be as in proposition \ref{formX}. We have
\begin{equation}\label{Futakinul}
\forall 
\zeta \in \aaa_1,~~ \int_{\Delta^+} \langle p + 2 \rho_P, \zeta \rangle e^{\langle -2p -4 \rho_p, \xi \rangle} \prod_{\alpha \in  \Phi^+_Q} ( \alpha, p  ) \, dp =0.
\end{equation}
\end{proposition}
\begin{proof}
We will compute $ F_X (Y) $ when $ Y = \zeta - \sqrt{-1} J \zeta $ where $\zeta \in \aaa_1$. We then have
\begin{align*}
F_X(Y)&= \int_M \tilde{\theta}_Y e^{\tilde{\theta}_X} \omega_g^n 
= \int_{G/H} \tilde{\theta} _Y e^{\tilde{\theta}_X} \omega_g^n,
\end{align*}
since $G/H$ is dense in $M$. We know that $G/H$ is a fibration over $G/P$ with fiber $P/H$ (see proposition \ref{PHt}). By $K$-invariance, by equations \eqref{omegantildehoro} and \eqref{omegansr}, we obtain, by doing an integration along the fiber (see proposition 6.15 of \cite{bott} for example), that there is a constant $C>0$ independent of $\xi$ and $\zeta$ such that,
\begin{align*}
F_X(Y) &= C \int_{P/H} (\nabla u , \zeta) e^{(\nabla u , \xi)} \prod_{\alpha \in \Phi^+_P} \langle \alpha, \nabla u + 4 t_{\rho_P} \rangle \MA_\RR(u) \, s_r^{-1}
\wedge \overline{s_r}^{-1}. \\
\end{align*}
By the polar decomposition (see proposition \ref{polar}), we obtain that there are nonzero constants $C_0,C_1,C_2,C_3$ independent of $\xi$ and $\zeta$ such that
\begin{align*}
F_X(Y) &= C_0 \int_{K} \left( \int_{\aaa_1} (\nabla u(x), \zeta) e^{(\nabla u(x), \xi)} \prod_{\alpha \in \Phi^+_P} \langle \alpha, \nabla u(x) + 4 t_{\rho_P} \rangle \MA_\RR(u)(x) dx \right) dk \\ 
&= C_1 \int_{ x \in \aaa_1} (\nabla u(x), \zeta) e^{(\nabla u(x), \xi)} \prod_{\alpha \in \Phi^+_P} \langle \alpha, \nabla u(x) + 4 t_{\rho_P} \rangle \MA_\RR(u)(x) \, dx \\
&= C_2 \int_{p' \in 2 \Delta} \langle p', \zeta \rangle e^{\langle p', \xi \rangle} \prod_{\alpha \in \Phi^+_P} ( \alpha, p' + 4 \rho_P ) \, dp \\
&= C_3 \int_{ p \in \Delta^+} \langle p + 2 \rho_P, \zeta \rangle e^{\langle -2p -4 \rho_P, \xi \rangle} \prod_{\alpha \in  \Phi^+_Q} ( \alpha, p  ) \, dp. \\
\end{align*}
We obtain the penultimate equality by changing the variable $p'= d_xu$ thanks to propostion \ref{potconv} and by using the duality in particular the equality $ \langle \alpha, t_{\rho_P} \rangle = ( \alpha, \rho_P )$ (see equation \eqref{Dada2}). We use the change of variable $p'=-2p - 4 \rho_P$ and equation \eqref{polytopeDelta} in order to get the last equality.
\end{proof}

\subsection{Monge-Ampère equation in horospherical case}

Let $(X,g)$ be a Kähler-Ricci soliton i.e.
$$
\Ric(\omega_g)- \omega_g=\LLL_X \omega_g.
$$
By using the $\partial 
\overline{\partial}$-lemma, there is a unique function $\psi \in \C^\infty(M,\RR)$ modulo a additive constante such that
$$
\omega_g = \omega_{g^0} + \sqrt{-1} \partial \overline{\partial} \psi.
$$
Noting (see for example \cite{TZ2}) that
$
\tilde{\theta}_X(g)= \tilde{\theta}_X(g^0) + X(\phi).
$
Solving the Kähler-Ricci soliton equation is equivalent to finding a smooth function $\psi $ solution in local coordinates of the following Monge-Ampère equation :
\begin{equation}\label{MAES}
\left \{
\begin{array}{rll}
    \det(g^{0}_{i \overline{j}} + \psi_{i \overline{j}} ) &= \det (g^{0}_{i \overline{j}}) \exp( h - \tilde{\theta}_X(g^0) - X(\psi) - \psi) \\
     (g^{0}_{i \overline{j}} + \psi_{i \overline{j}} )&>0, \\
     \end{array}  
  \right.
\end{equation}
where $h=h_{g^0}$ is the only function belonging to $\C^{\infty}(M,\RR)$ satisfying
$$
\Ric(\omega_g^{0}) - \omega_g^{0} = \sqrt{-1} \, \partial
\overline{\partial} h,
~~
\int_M e^{h} \omega^{n}_{g^0} = \int_M \omega^{n}_{g^0},
$$
The equation \eqref{MAES} can also be written globally in the form of 
\begin{equation}\label{MAE2}
\left \{
\begin{array}{rll}
   (\omega_{g^0} + \sqrt{-1} \partial \overline{\partial} \psi)^n &= e^{h-\psi - \tilde{\theta}_X(g^0) -X(\psi)} \omega_{g^0}^n. \\
      \omega_{g^0} + \sqrt{-1} \partial \overline{\partial} \psi&>0. \\
     \end{array}  
  \right.
\end{equation}
Now suppose that $ g^0 $ is an $K$-invariant Kähler metric on $M$. First of all, we know (see the proposition \ref{formX}) that the solitonic vector field $X$ is written as $X = \xi - \sqrt{-1} J \xi$ with $\xi \in \aaa_1$. Since the form $ \omega_g ^0 $ is $ K $-invariant, we also want to find a  $ K $-invariant  solution $ \psi $. Thanks to the density of $ G / H $ in $M$, we can reduce our study to this homogeneous space. Moreover, per $ K $-invariance, using the polar decomposition and the equation \eqref{omegantildehoro}, we can simply calculate this equation for the values $ \exp (x) H $ for $ x \in \aaa_1 $. By noting $u$ and $u^0$ the potentials defined by the equation \eqref{potudef} for the metrics $q$ and $q^0$ on the line bundle $K_M^{-1}$ where $q$ and $q^0$ are metrics such as $\omega_{q^0}=\omega_{g^0}$ and $\omega_{q}=\omega_{g}$, we have, since the function $\psi$ is $K$-invariant,
\begin{equation}\label{Olivier1}
\forall x \in \aaa_1,~~ u(x)=u^0(x) + \psi(\exp(x)H).
\end{equation}
Thus, the equation \eqref{MAES} is written on $P/H$, using the equation \eqref{omegantildehoro}, in the form
\begin{align*}
\prod_{\alpha \in \Phi^+_P} \langle \alpha , \nabla u(x) + 4 t_{\rho_P} \rangle & \cfrac{\MA_\RR(u)(x)}{4^r \, 2^{\operatorname{Card}(\Phi^+_P)}} \cdot \Omega = \\
& e^{h- \psi - (\nabla u(x),  \xi )} \prod_{\alpha \in \Phi^+_P} \langle \alpha , \nabla u^0(x) + 4 t_{\rho_P} \rangle \cfrac{\MA_\RR(u^0)(x)}{4^r \, 2^{\operatorname{Card}(\Phi^+_P)}} \cdot \Omega. 
\end{align*}
We can simplify this expression. To do this, we use the formulas \eqref {renormal} and \eqref{omegansr} which are written on $P/H$ in the form
$$
e^h \, \prod_{\alpha \in \Phi^+_P} \langle \alpha , \nabla u^0(x) + 4 t_{\rho_P} \rangle \cfrac{\MA_\RR(u^0)(x)}{4^r \, 2^{\operatorname{Card}(\Phi^+_P)}} \cdot \Omega = e^{-u^0} \, \Omega.
$$
and since $u=u^0 + \psi$ we get
\begin{equation}
\MA_\RR(u)(x) \cdot \cfrac{\prod_{\alpha \in \Phi^+_P} \langle \alpha , \nabla u(x) + 4 t_{\rho_P} \rangle}{4^r \, 2^{\operatorname{Card}(\Phi^+_P)}} = e^{-u - (  \nabla u(x),  \xi )}.
\end{equation}

\subsection{Continuity method}

We now want to prove the existence of Kähler-Ricci solitons. To do this, we will use the continuity method. To begin, we introduce into the Monge-Ampère equation a parameter $t \in[0,1]$:
\begin{equation}\label{MAEt}
\left \{
\begin{array}{rll}
 \det(g^{0}_{i \overline{j}} + \psi_{i \overline{j}} ) &= \det (g^{0}_{i \overline{j}}) \exp( h - \tilde{\theta}_X - X(\psi) -t \psi) \\
     (g^{0}_{i \overline{j}} + \psi_{i \overline{j}} )&>0. \\
     \end{array}  
  \right.
 \end{equation}
 Note that the equation \eqref{MAES} is the previous equation with $ t = $1. In addition, if a solution exists at time $ t $, we denote it by $ \psi_t $. Since $ \omega_{g^{0}} + \sqrt{-1} \, \overline{\partial} \partial \, \psi_t $ defines a $K$-invariant kählerian metric, it has a convex potential $ u_t $ defined by the equation \eqref{potudef} for a metric $q_t$ on $K_M^{-1}$ such that $\omega_{q_t} = \omega_{g_t}$. By making the same simplification as before, the first equation of \eqref{MAEt} can be written as
\begin{equation}\label{Alice1}
\MA_\RR(u_t)(x) \cdot  \cfrac{\prod_{\alpha \in \Phi^+_P} \langle \alpha , \nabla u(x) + 4 t_{\rho_P} \rangle }{4^r \, 2^{\operatorname{Card}(\Phi^+_P)}} = \exp \left[  - \left( u^0(x) +t \, u(x) \right) - ( \nabla u(x),  \xi ) \right].
\end{equation}

Finally, by denoting $ w_t = t \cdot u_t + (1-t) \cdot u ^ 0$, we have $u_t = u^0 + \psi_t$ and $w_t= u^0 + t \psi_t$, 
so we get the following real Monge-Ampère equation :
\begin{equation}\label{MAERt1}
\MA_\RR(u_t)(x) \cdot  \cfrac{\prod_{\alpha \in \Phi^+_P} \langle \alpha , \nabla u_t(x) + 4 t_{\rho_P} \rangle }{4^r \, 2^{\operatorname{Card}(\Phi^+_P)}} = \exp \left[ -w_t(x) - ( \nabla u_t(x),  \xi ) \right].
\end{equation}

Recall that the continuity method consists in considering the set $ S $ of times when there is a solution:
$$
S:=\lbrace t \in [0,1] ~:~ \text{there is a solution $\psi_t$ to the equation \eqref{MAEt} at time $t$} \big \rbrace,
$$
and to show that $ S $ is a non-empty closed and open set of $[0.1] $. The openness and the existence of a solution at $t=0$ come from the study of the Monge-Ampère equations made in \cite{Aub,Y}. You can also consult \cite{TZ1} for a study done in the case of the Kähler-Ricci solitons. Moreover, thanks to the Arzelà-Ascoli theorem, it is sufficient to have an a priori $\C^k$-estimate for any $k \geq 3$ for the potential $\psi_t$ to obtain the closeness (see for example section 3.1 of \cite{sze2} for more details). However, thanks to the work of Yau and Calabi done in Appendix A of \cite{Y}, we can reduce this problem to finding an $\C^0$-estimate. In addition, by the following Harnack inequality type (\cite{TZ1,WZ} for more details):
$
-\inf_M \psi_t \leq C (1+ \sup_M \psi_t),
$
a uniform upper bound for $\psi_t$ for times $t \in[0,1]$ is necessary. 

\subsection{Proof of the a priori estimate}

The approach used in this section is based on that of \cite{WZ}. We must prove an a priori estimate for $ t \in[0,1] $. Now, using the fact that $ 0 \in S $ and $ S $ is open, we can reduce this to proving the a priori estimate on $[t_0,1] $ for $ t_0> 0 $. Such a $ t_0 $ is fixed for the rest of the section. In addition, for simplicity, we note $r$ the real dimension of the vector space $\aaa_1$ and we choose a base $(a_1,\cdots,a_{r})$ of $\aaa_1$ with $(x_1,\cdots,x_r)$ the associated coordinates in $\aaa_1$. So we identify $\aaa_1$ with $\RR^r$. We note $\Vert \cdot \Vert$ the usual Euclidean standard on $\RR^r$. We start with a lemma that will be useful to later on.
\begin{lemma}\label{intw=0}
We have 
$$
\forall \zeta \in \aaa_1,~~ \int_{\aaa_1} \cfrac{\partial w_t}{\partial \zeta} \, e^{-w_t}dx=0.
$$
\end{lemma}
\begin{proof}
It is sufficient by linearity to show:
$$
\forall i=1,\cdots,r,~~\int_{\RR^{r}} \cfrac{\partial w_t}{\partial x_i} \, e^{-w_t}dx=0.
$$
We can write thanks to Fubini theorem:
\begin{align*}
\int_{\RR^{r}} \cfrac{\partial w_t}{\partial x_i} \, e^{-w_t}dx &= -\int_{\RR^{r}} \cfrac{\partial e^{-w_t}}{\partial x_i}dx 
=  
\int_{\RR^{r-1}} \left( \int^{+ \infty}_{-\infty} \cfrac{\partial e^{-w_{t,i}}}{\partial x_i} \, dx_i \right) dx_1 \cdots \widehat{dx_i} \cdots dx_n \\
&= \int_{\RR^{r-1}} \left[ e^{-w_{t,i}} \right]^{+ \infty}_{- \infty} dx_1 \cdots \widehat{dx_i} \cdots dx_n
\end{align*}
where $w_{t,i} : s \in \RR \mapsto w_t( x_1, \cdots, x_{i-1}, s, x_{i+1}, \cdots,x_n) \in \RR$ and where the other coordinates $x_k$ are fixed. To conclude, it is enough to show that
$$
\lim_{ s \rightarrow \pm \infty} e^{-w_{i,t}(s)}=0.
$$
To prove this, we notice that, by definition of the function $w_t$ and thanks to proposition \ref{potconv} applied to $u^0$ and $u$, we have
$$
e^{-w_t(x)} = \left(e^{-u^0} \right)^{1-t} \cdot \left( e^{-u} \right)^t \leq C e^{-v_{2 \Delta}(x)}, ~~\forall x \in \RR^{r}n
$$
where $C$ is an independent constant of $x$ and $v_{2 \Delta}$ is the support function of $2 \Delta$ i.e. $v_{2 \Delta}(x)= \sup_{p \in 2 \Delta} (x,p)$.
So we have for any $p \in 2 \Delta$
\begin{align*}
v_{2 \Delta}(x) &\geq (x,p)
=  x_i (a_i,p) + \sum_{j=1, j \neq i} x_j (a_j,p)  
 \geq x_i (a_i,p) + \inf_{p \in 2 \Delta} \left( \sum_{j=1, j \neq i} x_j (a_j,p) \right). 
\end{align*}
Finally, we get
$$
\forall p \in 2 \Delta,~~ 0 \leq e^{-w_{t,i}(s)} \leq \tilde{C} e^{-s (a_i,p)},
$$
where $\tilde{C}$ is an independent constant of $t$. To conclude, it is enough to notice that, thanks to the fact that $0 \in 2 \Delta$, there is a ball centered in $0$ with a radius $\delta >0$ contained in $2 \Delta$ and therefore there is $p_1 \in 2 \Delta$ such that $(a_i,p_1)>0$ and $p_2 \in 2 \Delta$ such that $(a_i,p_2)<0$.
\end{proof}
\begin{lemma}\label{wtext}
The function $w_t$ obtains its minimum $m_t$ at a point $x_t \in \aaa_1$ .
\end{lemma}
\begin{proof}
The demonstration is based on the fact that a convex function on $ \RR^{r} $ that admits a critical point admits a global minimum. To apply this result, we notice that $ w_t $ is a convex function because it is the barycenter of the convex functions $ u $ and $ u^0 $ (see the proposition \ref{potconv}). To conclude, it is therefore sufficient to show that $ 0 \in \nabla w_t (\RR^n) $. To this end, we note that
\begin{align*}
 \nabla w_t(\aaa_1)&= t \nabla u(\aaa_1) + (1-t) \nabla u^0 (\aaa_1) = 2\, t \operatorname{Int}(\Delta) +2 \, (1-t) \,  \operatorname{Int}(\Delta) = 2 \operatorname{Int}(\Delta) 
 \end{align*}
but we have $0 \in 2 \operatorname{Int}(\Delta)$ (see equation\eqref{0inD})
\end{proof}
\begin{lemma}\label{mt<C}
We have :
$$
\exists C>0, \forall t \in [t_0,1], m_t \leq C.
$$
\end{lemma}
Before starting the proof, we recall a result concerning convex domains that will be useful to us.
\begin{lemma}[\cite{Mi}]\label{ellip}
Let  $\Omega$ be a convex bounded domain of $\RR^n$. Then there is a unique ellipsoid $E$, called the minimum ellipsoid of $\Omega$, whose volume is minimal among ellipsoids containing $\Omega$. In addition, $E$ satisfies
$$
\cfrac{1}{n} E \subset \Omega \subset E.
$$
Let $T$ be an affine transformation where the vector part has determinant egal to $1$, preserving the center $x_0$ of $E$ i.e. $T(x)=T_{vect}(x-x_0)+x_0$ for a matrix $T_{vect} \in \operatorname{GL_n}(\RR)$ and such that $T(E)$ is a ball $B_R$ with center $x_0$ for a certain $R>0$ (depending on $E$). In particular, we have $B_{R/n} \subset T(\Omega) \subset B_R$.
\end{lemma}
\begin{proof}
We define, for all $k \in \NN$
$$
A_k := \lbrace x \in \RR^r ~/~ m_t + k \leq w_t(x) \leq m_t + k +1 \rbrace,
$$
and we have the following elementary properties:
\begin{itemize}
\item[$\bullet$] the set $A_k$ is bounded for $k \geq 0$ and $\bigcup_{k \in \NN} A_k = \aaa_1$
\item[$\bullet$] $x_t \in A_0$.
\item[$\bullet$] $\bigcup_{i=0}^k A_i$ is convex for every $k \geq 0$.
\end{itemize}
In addition, since $u_t$ and $u^0$ are convex, the matrices $((u_t)_{ij})_{1 \leq i,j \leq r }$ and $(u^{0}_{ij})_{1 \leq i,j \leq r }$ are positive and so we get (this is a direct consequence of Minkowski's formula, see for example page 115 of \cite{marcus1992survey})
$$
\det( (w_t)_{ij})= \det( t \, (u_t)_{ij} + (1-t)u^0_{ij}) \geq  \det( (u_t)_{ij})  + \det ((1-t)u^0_{ij}).
$$ 
Since $\det ((1-t)u^0_{ij}) \geq 0$, we get, using equation \eqref{MAERt1},
\begin{align*}
\det( (w_t)_{ij}) & \geq  \det( t  (u_t)_{ij}) = t^{r} \, \det( (u_t)_{ij})  = \frac{e^{-w_t(x)- (\nabla u_t(x), \xi )}}{ \prod_{\alpha \in \Phi_P^+} \langle \alpha, \nabla u_t(x) + 4 t_{\rho_P} \rangle}  \left(4^r 2^{\operatorname{Card}(\Phi^+_P)} \right) \\ 
& \geq t^{r} \, \cfrac{4^r 2^{\operatorname{Card}(\Phi^+_P)}}{d'} \,  e^{d} \, e^{-w_t} & \\
\end{align*}
where
$$
d= \inf \left\lbrace ( p, \xi ) ~/~ p \in -2 \Delta^+ - 4 \rho_P \right\rbrace \in \RR 
$$
and, thanks to the equation \eqref{f>0},
$$
0 < d'= \sup_{ p \in -2 \Delta^+} \left\lbrace \prod_{\alpha \in \Phi^+_P} \langle \alpha , p \rangle \right\rbrace < + \infty.
$$
Let $T$ and $R$ be as in lemma \ref{ellip} for $ \Omega=A_0 $ so that
\begin{equation}\label{bluff1}
B_{R/r} \subset T(A_0) \subset B_R.
\end{equation}
and so
\begin{equation}\label{Lea2}
\det( (w_t)_{ij})  \geq  C_0 e^{-m_t} ~~ \text{dans $T(A_0)$}.
\end{equation}
We want to prove
\begin{equation}\label{Paul2}
R \leq \sqrt{2} r C_0^{-1/2r} e^{m_t/2r}.
\end{equation}
Indeed, we define
$$
v : y \in \aaa_1   \longmapsto   \cfrac{1}{2} \, C_0^{1/r} e^{-m_t/r} \left[ \Vert y - y_t \Vert^2 - \left(\cfrac{R}{r} \right)^2 \right] + m_t +1 \in \RR
$$
where $y_t$ is the center of the minimal ellipsoid of $A_0$. A direct computation gives
$
\det(v_{ij})= C_0 e^{-m_t} ~~ \text{ in $T(A_0)$,}
$
and, by using equation \eqref{Lea2}, we get
$
\det(v_{ij}) \leq \det((w_t)_{ij}) \text{ in $ T(A_0)$ },
$
and
$
v(y) \geq m_t +1 \geq w_t(y)  \text { sur $ \partial T(A_0)$}.
$
We then apply the principle of comparison for the real Monge-Ampère equations (see for example \cite{gutierrez}) and we obtain
$
v \geq w
$
sur $T(A_0)$. In particular, we get
$$
m_t  \leq v(y_t) = - \cfrac{1}{2}C_0^{1/r}e^{-m_t/r} \left(\cfrac{R}{r} \right)^2 + m_t+1.
$$
Now, thanks to the convexity of $w$, we get
$$
A_k \subset \bigcup_{i=0}^k A_k \subset (k+1) \cdot A_0,
$$
where $(k+1) \cdot A_0$ is the dilation of $A_0$ of factor $(k+1)$. Moreover, thanks to the equation \eqref{bluff1} and the fact that $w_t$ is convex, we obtain 
$$
T(A_k) \subset  T( ((k+1) \cdot A_0)) \subset (k+1) \cdot T(A_0) = B_{(k+1)R}.
$$
Now, if we note $\omega_{r}$ the volume of the unit ball of $\RR^r$ then, using the equality $w_t \geq m_t + k$ on $A_k$, the fact that $T$ preserves the formula, the equality $T(A_k) \subset B_{(k+1)R})$ and the equation \eqref{Paul2}, we get
\begin{align*}
\int_{\aaa_1} e^{-w_t} & \leq \sum_k \int_{A_k} e^{-w_t} 
 \leq \sum_k e^{-m_t -k} \, \operatorname{Vol}(A_k)  = \sum_k e^{-m_t -k} \, \operatorname{Vol}(T(A_k)) \\
 & \leq  \omega_{r} \sum_k e^{-m_t -k} \left( (k+1)R \right)^{r} 
= \omega_{r}  \cfrac{(R)^{r}}{e^{m_t}} \sum_k \cfrac{(k+1)^{r}}{e^k}  \leq C_1 e^{-m_t/2},
\end{align*}
where $C_1>0$ is a time-independent constant. Finally, we get
\begin{align*}
e^{-m_t/2} & \geq \cfrac{1}{C_1} \int_{\aaa_1} e^{-w_t}.
\end{align*}
In addition, thanks to the equation \eqref{MAERt1} and the computation performed at the end of the demonstration of the proposition \ref{Paul6}, we have
\begin{align*}
e^{-m_t/2} & \geq \cfrac{1}{C_1}  \int_{\aaa_1} \MA_\RR(u_t)  e^{ ( \nabla u_t(x), \xi ) }\cfrac{\prod_{\alpha \in \Phi^+_P} \langle \alpha , \nabla u_t(x) + 4 t_{\rho_P} \rangle }{4^r \, 2^{\operatorname{Card}(\Phi^+_P)}} \, dx \\ &= \cfrac{1}{C_1}  \int_{\Delta^+} e^{\langle -2p- 4 \rho_P , \xi \rangle} \cfrac{ \prod_{\alpha \in \Phi^+_Q} ( \alpha , p ) }{4^r \, 2^{\operatorname{Card}(\Phi^+_P)}} \, dp =: \frac{C_2}{C_1} \\
\end{align*}
where $C_2$ is independent of the time $t$ and strictly positive since $\Delta^+$ has non-empty interior. Finally, we obtain a constant $C>0$ independent of $t$ such that
$
m_t \leq C.
$
\end{proof}
\begin{lemma}\label{mt>C}
We have 
$$
\exists \tilde{C}>0, \forall t \in [t_0,1],~~ m_t \geq \tilde{C}.
$$
\end{lemma}
\begin{proof}
Here, the proof is not based on \cite{WZ} but on \cite{Donaldson} also used in proposition 6.19 of \cite{Delcroix}. First of all, recall that $\Vert \nabla w_t \Vert \leq d_0:= \sup \lbrace \Vert x \Vert ~:~  x \in 2 \Delta \rbrace$. Thus, by the mean value theorem, we have
$$
\forall x \in \aaa_1,~~  \vert w_t(x) - m_t \vert \leq d_0 \, \Vert x - x^t \Vert.
$$
Hence, for all $x \in B(x^t,\frac{1}{d_0})$, we get
$$
 \vert w_t(x) - m_t \vert \leq 1,
$$
so that $x \in A_0$ and
$$
B\left(x^t,\frac{1}{d_0} \right) \subset A_0. 
$$
And we get
\begin{equation}\label{dernierespoir1}
\int_{A_0} dx = \operatorname{Vol}(A_0) \geq \int_{B\left(x^t,\frac{1}{d_0} \right)} dx = \operatorname{Vol}(B\left(x^t,\frac{1}{d_0} \right)) = c,
\end{equation}
where $c$ is time independent since $\operatorname{Vol}(B(x^t,\frac{1}{d_0}) = \operatorname{Vol}(B(0,\frac{1}{d_0})$. Recall, thanks to the previous computation, that there is a time independent constant $C_2$  such that
$$
C_2= \int_{\aaa_1} e^{-w_t}dx .
$$
We get therefore
\begin{align*}
C_2 &= \int_{\aaa_1 } e^{-w_t}dx = \int_{\aaa_1} \int_{w_t(x)}^{+ \infty} e^{-s} ds dx = \int_{- \infty}^{+ \infty} e^{-s} \left( \int_{\aaa_1}  {1}_{ \lbrace w_t \leq s \rbrace} dx \right) ds \\ &=  \int_{-m_t}^{+ \infty} e^{-s} \operatorname{Vol}(\lbrace w_t \leq s \rbrace) ds  =  e^{-m_t} \int_{0}^{+ \infty} e^{-s} \operatorname{Vol}(\lbrace w_t \leq s \rbrace) ds  \\ & \geq e^{-m_t} \int_{1}^{+ \infty} e^{-s} \operatorname{Vol}(A_0) \, ds.
\end{align*}
Hence, by using the equation \eqref{dernierespoir1}, we get
$$
C_2 \geq e^{-m_t} \, c \, \int_{1}^{+ \infty} e^{-s} ds,
$$
and 
$$
\ln(C_2) \geq -m_t + \ln \left( c \, \int_{1}^{+ \infty} e^{-s} ds \right).
$$
\end{proof}
\begin{lemma}\label{xt<C}
Let $x^t=(x_1^t, \cdots, x_n^t)$ be the point where $w_t$. obtains its minimum. There exists a constant C independant of time $t$ such that
$$
\Vert x^t \Vert \leq C'.
$$
\end{lemma}
\begin{proof}
For a contradiction, assume that
$$
\forall C'>0,~~ \exists t \in [ t_0,1],~~ \text{$\psi_t$ is defined and } \Vert x^t \Vert >C'.
$$
For future use, it should be noted that this implies in particular that $\Vert x^t \Vert \neq 0$. Recall that thanks to the previous computation, we have
$$
\int_{\aaa_1} e^{-w_t}dx = C_2>0.
$$
Recal also that $\Vert \nabla w_t \Vert \leq d_0:= \sup \lbrace \Vert x \Vert ~:~ x \in 2 \Delta \rbrace$ if there is a radius $R'>0$ independent of $t$ such that $\inf \lbrace w_t(x) ~:~ x \in \partial B(x^t,R') \rbrace \geq m_t +1$. Now, by convexity, we have
$$
\forall x \in \aaa_1 \setminus B(x^t,R'),~~ w_t(x) \geq \cfrac{1}{R'} \Vert x - x^t \Vert + m_t.
$$
Hence for all $\varepsilon>0$, there is $\delta \geq R'$ independent of $t$ such that, by lemma \ref{mt>C},
\begin{equation}\label{Paul4}
\int_{\aaa_1 \setminus B(x^t,\delta)} e^{-w_t(x)}dx \leq e^{-\tilde{C}} \int_{\aaa_1 \setminus B(x^t,\delta) } e^{- \frac{1}{R'} \, \Vert x-x^t \Vert}dx \leq \varepsilon.
\end{equation}
We fix $ \varepsilon $ and $ \delta $ which satisfies the above property. We will use an argument from \cite{Delcroix2} (used in particular in this theorem 6.30) rather than the original argument of \cite{WZ}. Since $u^0$ is a convex function asymptote to function $v_{2 \Delta}$ and $ \nabla u^0 $ is a diffeomorphism of $ \aaa_1 $ in $ \operatorname{Int}(2 \Delta) $ and $ 0 \in \operatorname{Int}(2 \Delta) $( see proposition \ref{roro1}), by inequality of convexity, we obtain that
$$
\forall x \in B(x^t,\delta),~~\cfrac{\partial u^0 }{\partial \zeta}(x) \geq \cfrac{1}{2} \, a_0, 
$$
where $\zeta= x^t / \Vert x^t \Vert$ and $a_0= \inf \lbrace \vert v_{2 \Delta}(\xi) \vert ~:~ \xi \in \aaa, \Vert \xi \Vert=1 \rbrace > $0. We obtain, thanks to the equation \eqref{Paul4} and reducing $\epsilon$ in the last inequality if necessary,
\begin{align*}
\int_{B(x^t,\delta)} \cfrac{\partial u^0 }{\partial \zeta}(x) e^{-w_t} dx  &\geq \cfrac{1}{2} \, a_0 \int_{B(x_t,\delta)} e^{-w_t} \, dx  \geq \cfrac{1}{2} \, a_0 \left(  \int_{\aaa_1} e^{-w_t} \, dx - \int_{\aaa_1 \setminus B(x^t,\delta)} e^{w_t(x)} \, dx \right) \\ & \geq \cfrac{1}{2} \, a_0 \left(  \int_{\aaa_1} e^{-w_t(x)} \, dx - \varepsilon \right) \geq \cfrac{1}{4} \, a_0 C_2 .
\end{align*}
so for $\varepsilon$ small enough,
\begin{equation}\label{Paul7}
\int_{\aaa_1} \cfrac{\partial u^0 }{\partial \zeta}(x) e^{-w_t} dx >0.
\end{equation}
Now, we will prove that for all $ \zeta \in \aaa_1$, we have
\begin{equation}\label{Paul5}
\int_{\aaa_1} \cfrac{\partial u^0}{\partial \zeta} \, e^{-w_t}dx=0.
\end{equation}
Indeed, the proposition \ref{Paul6} and this proof give a constant $C>0$ such that
\begin{align*}
0 &= \int_{\aaa_1} ( \nabla u_t(x), \zeta ) e^{- ( \nabla u_t(x), \xi) } \prod_{\alpha \in \Phi^+_P} \langle \alpha, \nabla u_t(x) + 4 t_{\rho_P} \rangle\MA_\RR(u_t)(x) dx  \\& = C_0' \int_{\aaa_1} ( \nabla u_t(x), \zeta ) e^{-w_t(x)}dx 
 = C_0' \int_{\aaa_1} \cfrac{\partial u_t}{\partial \zeta}(x) e^{-w_t(x)}dx \\ & = C_0' \cfrac{1-t}{t} \int_{\aaa_1}\cfrac{\partial u^0}{\partial \zeta}(x) e^{-w_t(x)}dx -  C_0' \cfrac{1}{t} \int_{\aaa_1}\cfrac{\partial w_t}{\partial \zeta}(x) e^{-w_t(x)}dx \\
&= C_0' \cfrac{1-t}{t} \int_{\aaa_1}\cfrac{\partial u^0}{\partial \zeta}(x) e^{-w_t(x)}dx.
\end{align*}
So we get 
\begin{equation*}
\forall \zeta \in \aaa_1,~~\int_{\aaa_1}\cfrac{\partial u^0}{\partial \zeta}(x) e^{-w_t(x)}dx=0.
\end{equation*}
This prove equation \eqref{Paul5}, which contradicts the equation \eqref{Paul7}. This concludes the proof.
\end{proof}
The proof of the a priori estimate is by concluded with the following lemma.

\begin{lemma}\label{psit<C}
Let $\psi_t$ be a solution of equation \eqref{MAEt} where $t \in [t_0,1]$. We have
$$
\sup_M \psi_t \leq C'',
$$
for a constant $C''$ independent of $t$.
\end{lemma}
\begin{proof}
By density and $K$-invariance, it is sufficient to prove that
$$
\sup_{ y \in \aaa_1} \psi_t (\exp(y)H) \leq C''.
$$
By convexity of $u_t$, we have
$$
\forall y \in \aaa_1,~~u_t(0) + ( \nabla u_t(y),y ) \geq u_t(y).
$$
By definition of $v_{2\Delta}$ and since $\nabla u_t(\aaa_1) =2 \operatorname{Int}(\Delta)$ by proposition \ref{roro1}, we have
$$
v_{2 \Delta}(y) + u_t(0) \geq u_t(y).
$$
Now, we have, thanks to \eqref{Olivier1} and because $v_{2 \Delta} - u_0$ is bounded,
\begin{align*}
\psi_t( \exp(y) H) &= u_t(y) - u_0(y)  
\leq v_{2\Delta}(y) + u_t(0) - u_0(y) 
\leq a + u_t(0).
\end{align*}
So it is sufficient to prove that $u_t(0)$ has upper bound independent of $t$.

As $\nabla w_t (\aaa_1) =  \operatorname{Int}( 2\Delta) $  which is a bounded polytope, we have
$
\vert \nabla w_t (\aaa_1) \vert  \leq d_0:= \sup \lbrace \Vert x \Vert ~:~ x \in 2 \, \Delta \rbrace,
$
and so
$
\vert w_t(0) - w_t(x^t) \vert \leq d_0 \Vert x^t \Vert.
$
Moreover, thanks to the lemma \ref{xt<C}, we get  $C'>0$ independant of $t$ such that $\Vert x^t \Vert \leq C'$ and so
$
\vert w_t(0) - w_t(x^t) \vert \leq d_0 C'.
$
By the lemma \ref{mt<C}, we get $w_t(x^t)=m_t \leq C$ where $C$ is a constant independant of $t$. So, we get
$
w_t(0) \leq C + d_0  C.
$
But $w_t=tu_t + (1-t)u^0$ so
$
 t u_t (0) \leq  C + d_0 C' -(1-t) u^0(0).
$
 Finally, as $ t \in [t_{0}, 1] $, we get
$
 u_t(0) \leq \Theta,
$
where $\Theta$ is a constant independant of $t$.
\end{proof}

\section{Kähler Einstein metric in the horospherical case}

In this section, we study under which necessary and sufficient conditions the Kähler-Ricci soliton $(X,g)$ is trivial i.e. $X=0$ and therefore when $g$ is a Kähler-Einstein metric. This part is based on \cite{Delcroix}.

\subsection{Condition of existence}

We have just proved that any Fano horospherical manifold admits a Kähler-Ricci soliton. This result can be completed by the following corollary.  
\begin{corollary}
Suppose that $M$ is a horospherical embedding of a homogeneous horospherical space $G/H$ such that $H$ contains the opposite Borel subgroup $B^-$ of $G$. Denote $\Delta^+$ the moment polytope of $X$ associated with the Borel subgroup $B$. Then $M$ admits a Kähler-Einstein metric if and only if
$$
\operatorname{Bar}(\Delta^+)_{DH} =-2 \rho_P,
$$
where $P:=N_G(H)$ and $\operatorname{Bar}_{DH}(\Delta^+)$ is the barycenter of the polytope $\Delta^+$ for the Duistermaat-Heckman measure.
\end{corollary}

\begin{proof}
We know that there is a soliton $(X,g)$. By the uniqueness of Kähler-Ricci soliton, it is therefore necessary and sufficient to prove that $X=0$ if and only if $\operatorname{Bar}(\Delta^+)_{DH} =-2 \rho_P$. In particular, thanks to the theoretical results concerning the Futaki invariant, it is therefore necessary to prove that $F_0$ is identically zero if and only if $\operatorname{Bar}(\Delta^+)_{DH} =-2 \rho_P$. Recall that $F_0$ is written for $ \chi \in \ggg^{1,0}$ (such that its real part is identified with the vector field defined by $x \mapsto \frac{d}{dt}_{t=0} \exp(t \text{Re}(\chi) )x$):
$$
F_0(\chi) =  \int_M \tilde{\theta}_\chi \, \omega_g^n. 
$$
where \begin{equation}
\LLL_\chi \omega_g = \sqrt{-1} \, \partial \overline {\partial} \, \tilde{\theta}_\chi,
~~
\int_M \tilde{\theta}_\chi e^{h_g} \omega^n_g = 0,
\end{equation}
and where the function $h_g \in \C^\infty(M,\RR)$ (defined modulo a additive constant) satisfies
\begin{equation*}
\Ric(\omega_g) - \omega_g = \sqrt{-1} \partial
\overline{\partial} h_g.
\end{equation*}
Moreover, we know that we restrict our study on $\ppp/\hhh \simeq \zzz(\ggg)$. We then obtain, thanks to the equation \eqref{Futakinul}, for every $\zeta \in \aaa_1$ :
\begin{align*}
F_0(\zeta)&= \tilde{M} \int_{\Delta^+} \langle p+2 \rho_P, \zeta \rangle \prod_{\alpha \in \Phi^+_Q} ( \alpha, p ) \, dp,\\
\end{align*}
where $\tilde{M}$ is an constant independent of $\zeta$. We have that $g$ is therefore a Kähler-Einstein metric if and only if
$$
 \forall \zeta \in \aaa_1,~~\int_{\Delta^+} \langle p+2 \rho_P, \zeta \rangle \prod_{\alpha \in \Phi^+_Q} ( \alpha, p ) \, dp =0.
$$
This equation can be written in the form
$$
\langle \operatorname{Bar}_{DH}(\Delta^+)+ 2 \rho_P, \zeta \rangle = 0.
$$
So $g$ is a Kähler-Einstein metric if and only if
$$
\operatorname{Bar}_{DH}(\Delta^+)=-2 \rho_P.
$$
\end{proof}
\subsection{Monge-Ampère Equation and greatest Ricci lower bound}
We keep the notations introduced in the part \ref{setuphoro}. To prove the existence of Kähler-Einstein metrics via the continuity method, it is sufficient to take $X=0$ in the equation \eqref{MAEt} and $\xi=0$ in the equation \eqref{MAERt1}. We get the following complex Monge-Ampère equation : \begin{equation}\label{MAEKEt}
\left \{
\begin{array}{rll}
 \det(g^{0}_{i \overline{j}} + \psi_{i \overline{j}} ) &= \det (g^{0}_{i \overline{j}}) \exp( h -t \psi) \\
     (g^{0}_{i \overline{j}} + \psi_{i \overline{j}} )&>0 \\
     \end{array}  
  \right.
 \end{equation}
ans so the real Monge-Ampère equation :
\begin{equation}\label{MAERt}
\MA_\RR(u_t)(x) \cdot  \prod_{\alpha \in \Phi^+_P} \langle \alpha , \nabla u_t(x) + 4 t_{\rho_P} \rangle  \, = \exp \left[ -w_t(x) \right],
\end{equation}
where $w_t = t \, u_t + (1-t) \, u^0$ (recall that $u^0$ is the potential defined by the equation \eqref{potudef} for the reference metric $g^0$). Since this definition is unchanged, the lemmas demonstrations \ref{intw=0}, \ref{wtext}, \ref{mt<C} and \ref{psit<C} may apply to this case. In particular, we therefore obtain the following proposition.
\begin{proposition}
There is a solution to time $t_0>0$ to the equation \eqref{MAEKEt} if there is a constant $C>0$ independent of time $t$ such that
$$
\sup_{t \in[0,t_0[} ~~\vert x^t \vert < C,
$$
where $x^t$ is the point where the minimum of the $w_t$ function is achieved.
\end{proposition}
With this result, we obtain the following caracterisation of the greatest Ricci lower bound
\begin{corollary}
The greatest Ricci lower bound is equal to
$$
R(M) = \sup \left\lbrace t_0 \in[0.1] ~:~ \exists C>0,~~ \forall t < t_0,~~ \vert x^t \vert < C \right\rbrace.
$$
\end{corollary}

\subsection{Computation of the greatest Ricci lower bound}
We will compute $R(M)$ in the horospherical case based on the work of \cite{Delcroix}. Before we begin, let's note that, thanks to the equation \eqref{MAERt},
\begin{equation}\label{Francine1}
\int_{\aaa_1} e^{-w_t(x)}dx = V,
\end{equation}
where $V$ is the volume of $\Delta^+$ for Duistermaat-Heckman measure.

\begin{theorem}
Suppose that $M$ is a horospherical embedding of a homogeneous horospherical space $G/H$ such that $H$ contains the opposite Borel subgroup $B^-$ of $G$. Note $P=N_G(H)$ and $\Delta^+$ the moment polytope of $X$ with respect to the Borel subgroup $B$.  In addition, it is assumed that $2 \rho_P \neq \operatorname{Bar}_{DH}(\Delta^+)$ where $\operatorname{Bar}_{DH}(\Delta^+)$ is the barycentre of the polytope $\Delta^+$ for the Duistermaat-Heckman measurement. We then have that $R(M)$ is the only $t \in \in \, ]0.1[$ such that
$$
\dfrac{t}{t-1}(\operatorname{Bar}_{DH}(\Delta^+) + 2 \rho_P) \in \partial \left( \Delta^+  + 2\rho_P \right).
$$
\end{theorem}
\begin{proof}
Recall that we have
$$
\int_{\aaa_1} e^{-w_t}dx = V.
$$
Recall also that $\Vert \nabla w_t \Vert \leq d:= \sup \lbrace \Vert x \Vert ~:~ x \in 2 \Delta^+ \rbrace$ so there is a radius $R'>0$ independent of $t$ such that $\inf \lbrace w_t(x) ~:~ x \in \partial B(x^t,R) \rbrace \geq m_t +1$. Now, by convexity of $w_t$, we have
$$
 \forall x \in \aaa_1 \setminus B(x_t,R'),~~ w_t(x) \geq \cfrac{1}{R} \Vert x - x^t \Vert + m_t.
$$
So, for all $\varepsilon>0$, there is $\delta:=\delta_\varepsilon \geq R'$ independent of $t$ such that
\begin{equation}\label{aaacutb}
\int_{\aaa_1 \setminus B(x^t,\delta)} e^{-w_t(x)}dx \leq e^{-\tilde{C}}  \int_{\aaa_1 \setminus B(x^t,\delta) } e^{- \frac{1}{R'} \Vert x-x^t \Vert}dx \leq \varepsilon.
\end{equation}
Denote $t_\infty= R(M)$ which is stricly lower of 1 since $M$ does not admits Kähler-Einstein metrics.
Let's start by noting that 
$$ \lim_{t \rightarrow t_\infty} \Vert x^t \Vert = + \infty.
$$
Indeed, if $ \Vert x^t \Vert$ admits an adherence value then the equation \eqref{MAEt} admits a solution in $t=t_\infty$ and since all the set of solutions is open, there is a $\delta>0$ such that there is a solution in $t=t_\infty + \delta$. This contradicts the maximum of $R(M)=t_\infty$. In addition, by posing for any $t \in[0, t_\infty]$, $\xi_t := \frac{x^t}{\vert x^t \vert }$, we can find a sequence of $(t_i)_{i \in \NN}$ of $[0,t_\infty]$ such that $[0,t_i \rightarrow t_\infty$ and $\xi_{t_\infty} \in \aaa_1 $ verifying
$$
\lim_{i \rightarrow \infty} \xi_{t_i} = \xi_{t_\infty}. 
$$
Let's start by noting that for all  $t \in [0,t_\infty]$, we have
\begin{align*}
\int_{\aaa_1} (\nabla u_t, \xi_t)e^{-w_t}dx &= \int_{\aaa_1} (\nabla u_t, \xi_t) \MA_\RR(u_t) \cdot \prod_{\alpha \in \Phi_P^+} \langle \alpha, \nabla u_t + 4 \rho_P \rangle dx  \\ &= \int_{2\Delta} \langle p, \xi_t \rangle \cdot \prod_{\alpha \in \Phi_P^+} \left( \alpha, p + 4 \rho_P \right)  dp  \\ & = \int_{\Delta^+} - \langle p + 2\rho_P, \xi_t \rangle \cdot \prod_{\alpha \in \Phi_Q^+} \left( \alpha, p  \right)  dp \\ & = - \langle \operatorname{Bar}_{DH}( \Delta^+) + 2\rho_P, \xi_t \rangle V.
\end{align*}
We get, by passing to the limit, 
\begin{equation}\label{nablaut}
\lim_{i \rightarrow + \infty} \int_{\aaa_1} (\nabla u_{t_i}, \xi_{t_i})e^{-w_{t_i}}dx = - \langle \operatorname{Bar}_{DH}(\Delta^+) +2 \rho_P, \xi_{t_\infty} \rangle V.
\end{equation}
Recall that we defined the support function $v_{2 \Delta}$ by the equation \eqref{fctsupport}. With this definition, we want to show that 
\begin{equation}\label{limu0}
\lim_{i \rightarrow + \infty} \int_{\aaa_1} (\nabla u^0 , \xi_{t_i}) e^{-w_{t_i}} = v_{2 \Delta}(\xi_{t_\infty})V. 
\end{equation}
We fix $\varepsilon>0$. We define $\theta:= \varepsilon/6d$ where $d:=\sup_{x \in 2 \,\Delta} \Vert x \Vert$ and we have, by the formula \eqref{aaacutb} with $\delta=\delta_\theta$, that
$$
\int_{\aaa_1 \backslash B(x^t,\delta)} e^{-w_t} dx \leq \theta.
$$
So we get
\begin{align*}
\Big\vert \int_{\aaa_1 \backslash B(x^t, \delta)} (\nabla u^0, \xi_t) \, e^{w_t} dx \Big\vert & \leq \int_{\aaa_1 \backslash B(x^t, \delta)} \vert (\nabla u^0, \xi_t) \vert \,  e^{-w_t} dx   \leq d \int_{\aaa_1 \backslash B(x^t, \delta)} e^{-w_t} dx  \leq d \theta.
\end{align*}
That implies, by definition of $\theta$ :
\begin{equation}\label{morc1}
\Big\vert \int_{\aaa_1 \backslash B(x^t, \delta)} (\nabla u^0, \xi_t) \, e^{-w_t} dx \, \Big\vert \leq \varepsilon/6.
\end{equation}
Moreover, since $\nabla u^{0}(\aaa_1) = \operatorname{Int}(2 \Delta)$(proposition \ref{potconv}), we have
$$
\forall x \in B(x^t,\delta),~~(\nabla u^0, \xi_t ) \leq v_{2 \Delta}.(\xi_t).
$$
Since $u^0$ is a convex function, we have
$$
(\nabla u^0, \xi_t) \geq \dfrac{u^0(x) - u^0(x-x^t)}{\Vert x_t \Vert}.
$$
For all $x \in B(x^t,\delta)$, we have $x-x^t \in B(0, \delta)$ so there is a constant $C_0 \in \RR$ independent of $t$ such that 
$$
\forall x \in B(x^t,\delta),~~u^0(x-x^t) \geq C_0.
$$
Moreover, we have (proposition \eqref{potconv}) the existence of a constant $C_1>0$ time independent such that
$$
\forall x \in \aaa_1,~~-C_1 < u^0(x) - v_{2 \Delta}(x) < C_1.
$$
So by denoting $C=-C_0-C_1$, we get
$$
\forall x \in B(x^t,\delta),~~(\nabla u^0(x), \xi_t) \geq \dfrac{C+v_{2 \Delta}}{\Vert x^t \Vert}.
$$
Now we get, for all $x \in B(x^t,\delta)$,
\begin{align*}
(\nabla u^0(x), \xi_t) &\geq v_{2 \Delta}(\xi_t) 
- v_{2 \Delta} \left( -\cfrac{x-x^t}{\Vert x^t \Vert} \right) + \frac{C}{\Vert x^t \Vert} \\
& \geq v_{2 \Delta}(\xi_t) - \cfrac{\Vert x - x^t \Vert} {\Vert x^t \Vert } \, v_{2 \Delta} \left( \frac{x^t-x}{ \Vert  x^t-x \Vert } \right) + \frac{C}{ \Vert x^t \Vert} \\
& \geq v_{2 \Delta} (\xi_t) + \frac{
C}{\Vert x^t \Vert} - \frac{\delta}{ \Vert x^t \Vert } \,  \inf_{ \Vert y \Vert =1} v_{2 \Delta} (y).
\end{align*}
There is a constant $C'>0$ indépendent of $t$ such that
$$
\dfrac{C'}{\Vert x^t \Vert} \leq (\nabla u^0(x), \xi_t) - v_{2 \Delta}(\xi_t) \leq 0.
$$
In particular, since $\Vert x^t \Vert$ tends to infinity when $t \rightarrow t_\infty$, we get
$$
\exists i_0 \in \NN,~~ \forall i \geq i_0, ~~ \vert (\nabla u^0(x), \xi_t) - v_{2 \Delta}(\xi_{t_i}) \vert \leq \varepsilon/3 V,
$$
and so
\begin{equation}\label{morc2}
\big\vert \int_{B(x^t,\delta)} \left(\nabla u^0(x) - v_{2 \Delta}(\xi_t) \right)e^{-w_t(x)} dx \big\vert \leq \varepsilon/3.
\end{equation}
by using the equation \eqref{aaacutb} with $\delta=\delta_\theta$, we get

\begin{align*}
\Big\vert v_{2 \Delta} (\xi_t)V - \int_{B(x^t,\delta)} v_{2 \Delta}(\xi_t) \, e^{-w_t(x)} dx \, \Big\vert & = \Big\vert \int_{\aaa_1} v_{2 \Delta}(\xi_t) \, e^{-w_t(x)} dx - \int_{B(x^t,\delta)} v_{2 \Delta}(\xi_t) \, e^{-w_t(x)} dx  \Big\vert \\
&= \Big\vert \int_{\aaa_1 \backslash B(x^t,\delta)} v_{2\Delta}(\xi_t) \, e^{-w_t(x)} dx \Big\vert \leq d \theta.
\end{align*}
So we get
\begin{equation}\label{morc3}
\big\vert v_{2 \Delta}(\xi_t)V - \int_{B(x^t,\delta)} v_{2 \Delta}(\xi_t) \, e^{-w_t(x)} dx 
\big\vert \leq \varepsilon/6. 
\end{equation}
Moreover, by continuity of $v_{2 \Delta}$,
\begin{equation}\label{morc4}
\exists i_1 \in \NN,~~  \forall i >i_1,~~ \vert v_{2 \Delta}(\xi_{t_i})V - v_{2 \Delta}(\xi_{t_0})V \vert \leq \varepsilon/3.
\end{equation} 
We can conclude, thanks to \eqref{morc1}, \eqref{morc2}, \eqref{morc3}, \eqref{morc4}, that:
\begin{align*}
\forall i \geq \max(i_0,i_1),~~ \Big\vert \int_\aaa ( \nabla u^0(x), \xi_{t_i} ) e^{-w_{t_i}(x)}dx -v_{2 \Delta}(\xi_\infty)V \Big\vert & \leq \varepsilon.
\end{align*}
So we get the formula \eqref{limu0}.

Now, thanks to the equation \ref{intw=0}, we get
$$
0= \int_{\aaa_1}( \nabla w_t(x), \xi_t) \, e^{-w_t(x)} dx = t \, \int_{\aaa_1} (\nabla u_t, \xi_t ) \, e^{-w_t} dx + (1-t) \int_{\aaa_1} (\nabla u^0 , \xi_t) \, e^{-w_t(x)}dx.
$$
So, by passing to the limit, we get
$$
- t_\infty \, \langle \operatorname{Bar}_{DH}(\Delta^+) + 2 \rho_P, \xi_{t_\infty} \rangle + (1-t_\infty) \, v_{2 \Delta}(\xi_{t_{\infty}})  =0,
$$
and so
\begin{equation}\label{Marguerite}
- \frac{t_\infty}{t_\infty-1} \langle \operatorname{Bar}_{DH}(2 \Delta^+) +2 \rho_P , \xi_{t_\infty} \rangle = v_{2 \Delta}(\xi_\infty).
\end{equation}
To conclude, we study the function
$$
f : t \in [0,1[ \, \mapsto \, - \frac{t}{t-1} \left( \operatorname{Bar}(2 \Delta^+) \right)+2 \rho_P  \in \aaa_1^*.
$$
Note that $f(0)=0 \in \aaa_1^*$. Now we know that $0 \in 2 \operatorname{Int}(\Delta)$ so $f(0) \in \operatorname{Int}(\Delta)$. Thus, since the function $t/(t-1)$ is a strictly decreasing function on $[0,1[$ with values in $]-\infty,0]$, the values of $f$ run the half line of origin $0$ and direction $-(\operatorname{Bar}_{DH}(2 \Delta^+) + 2\rho_P)$. Thanks to the equation \eqref{Marguerite}, we obtain that
$
f(t_\infty)
$
is a vector such that $\langle f(t_\infty) , \xi_{t_\infty} \rangle = v_{2 \Delta}(\xi_{t_\infty})$, this means that $f(t_\infty)$ belongs to the support hyperplane defined by $\xi_{t_\infty}$. So $t_\infty$ is the only time $t \in[0,1[$ such that
$$
-\cfrac{t}{t-1} \, \left( \operatorname{Bar}_{DH}(2 \Delta^+) +2 \rho_P \right) \in \partial  \left( 2\Delta \right).
$$
By using the equation \eqref{polytopeDelta}, we can rewritte the last equation in the form
$$
\cfrac{t}{t-1} \, \left( \operatorname{Bar}_{DH}(\Delta^+) + 2 \rho_P  \right) \in \partial  \left( \Delta^+ + 2\rho_P \right).
$$
\end{proof}
\bibliographystyle{alpha}
\bibliography{biblio}

\begin{thebibliography}{KMoOG84}

\bibitem[Aub78]{Aub}
T.~Aubin.
\newblock \'{E}quations du type {M}onge-{A}mp\`ere sur les vari\'et\'es
  k\"ahl\'eriennes compactes.
\newblock {\em Bull. Sci. Math.}, 102:63 -- 95, 1978.

\bibitem[Bri87]{Bri87}
M.~Brion.
\newblock {\em Sur l'image de l'application moment}, pages 177--192.
\newblock Springer, 1987.

\bibitem[Bri89]{Bri89}
M.~Brion.
\newblock Groupe de {P}icard et nombres caract\'eristiques des vari\'et\'es
  sph\'riques.
\newblock {\em Duke Math. J.}, pages 397--424, 04 1989.

\bibitem[BT95]{bott}
R.~Bott and L.W. Tu.
\newblock {\em Differential Forms in Algebraic Topology}.
\newblock Graduate Texts in Math. Springer, 1995.

\bibitem[Del15]{Delcroix}
T.~Delcroix.
\newblock {\em {K{\"a}hler-{E}instein metrics on group compactifications}}.
\newblock PhD thesis, {Universit{\'e} Grenoble Alpes}, 2015.

\bibitem[{Del}16]{Delcroix2}
T.~{Delcroix}.
\newblock {K-Stability of Fano spherical varieties}.
\newblock {\em ArXiv e-prints}, (arXiv:1608.01852), August 2016.

\bibitem[Dem]{Demailly1}
J.~P. Demailly.
\newblock Complex analytic and differential geometry.
\newblock Notes de cours.

\bibitem[{Don}08]{Donaldson}
S.~K. {Donaldson}.
\newblock {K\"ahler geometry on toric manifolds, and some other manifolds with
  large symmetry.}
\newblock In {\em {Handbook of geometric analysis. No. 1}}, pages 29--75.
  Inter. Press, 2008.

\bibitem[Fuj78]{Fujiki1978}
A.~Fujiki.
\newblock On automorphism groups of compact {K}ähler manifolds.
\newblock {\em Invent.. math.}, 44:225--258, 1978.

\bibitem[Gau]{dercalabi}
Paul Gauduchon.
\newblock {C}alabi's extremal {K}ähler metrics : An elementary introduction.
\newblock
  http://germanio.math.unifi.it/wp-content/uploads/2015/03/dercalabi.pdf.

\bibitem[Gut01]{gutierrez}
C.E. Gutierrez.
\newblock {\em The Monge-Amp{\`e}re Equation}.
\newblock Birkh{\"a}user, 2001.

\bibitem[Guz75]{Mi}
M.~Guzm{\'a}n.
\newblock {\em Differentiation of Integral in $\RR^n$}.
\newblock Lecture Notes in Math. Springer, 1975.

\bibitem[Ham88]{Hamilton}
R.~S. Hamilton.
\newblock The {R}icci flow on surfaces.
\newblock In {\em Mathematics and general relativity ({S}anta {C}ruz, {CA},
  1986)}, volume~71 of {\em Contemp. Math.}, pages 237--262. Amer. Math. Soc.,
  1988.

\bibitem[{Hua}17]{2017arXiv170507735H}
H.~{Huang}.
\newblock {{K}ahler-{R}icci flow on homogeneous toric bundles}.
\newblock {\em ArXiv e-prints}, 2017.

\bibitem[KMoOG84]{K1}
F.C. Kirwan, J.N. Mather, University of~Oxford, and P.~Griffiths.
\newblock {\em Cohomology of Quotients in Symplectic and Algebraic Geometry}.
\newblock Princeton Univ., 1984.

\bibitem[Kno91]{Knop91}
Friedrich Knop.
\newblock The {L}una-{V}ust theory of spherical embeddings.
\newblock Notes de conférence, 1991.

\bibitem[Kob12]{Kobayashi2}
S.~Kobayashi.
\newblock {\em Transformation Groups in Differential Geometry}.
\newblock Springer, 2012.

\bibitem[MM92]{marcus1992survey}
M.~Marcus and H.~Minc.
\newblock {\em A Survey of Matrix Theory and Matrix Inequalities}.
\newblock Dover Publications, 1992.

\bibitem[Pas06]{pasquierthese}
B~Pasquier.
\newblock {\em {Fano horospherical varieties}}.
\newblock PhD thesis, {Universit{\'e} Joseph-Fourier - Grenoble I}, 2006.

\bibitem[Pas09]{Pasquier2009}
B.~Pasquier.
\newblock On some smooth projective two-orbit varieties with {P}icard number 1.
\newblock {\em Math. Ann.}, 344:963--987, 2009.

\bibitem[PS10]{MR2658183}
F.~Podest\`a and A.~Spiro.
\newblock {K}ahler-{R}icci solitons on homogeneous toric bundles.
\newblock {\em J. reine angew. Math.}, 642:109--127, 2010.

\bibitem[Ser56]{AIF_1956__6__1_0}
J.~P. Serre.
\newblock G\'eom\'etrie alg\'ebrique et g\'eom\'etrie analytique.
\newblock {\em Ann. de l'Inst. Fourier}, 6:1--42, 1956.

\bibitem[Spr98]{Spr98}
T.A. Springer.
\newblock {\em Linear Algebraic Groups}.
\newblock Modern Birkh{\"a}user Classics. Birkh{\"a}user, 1998.

\bibitem[Sze11]{sze}
G.~Szekelyhidi.
\newblock Greatest lower bounds on the {R}icci curvature of {F}ano manifolds.
\newblock {\em Compositio Math.}, 147:319–331, 2011.

\bibitem[Sz{\'e}14]{sze2}
G.~Sz{\'e}kelyhidi.
\newblock {\em An Introduction to Extremal K{\"a}hler Metrics}.
\newblock Amer. Math. Soc., 2014.

\bibitem[Tim11]{Tim11}
D.A. Timashev.
\newblock {\em Homogeneous Spaces and Equivariant Embeddings}.
\newblock Springer, 2011.

\bibitem[TZ00]{TZ1}
Gang Tian and Xiaohua Zhu.
\newblock Uniqueness of {K}ähler-{R}icci solitons.
\newblock {\em Acta Math.}, 184:271--305, 2000.

\bibitem[TZ02]{TZ2}
G.~Tian and X.~Zhu.
\newblock A new holomorphic invariant and uniqueness of {K}ähler-{R}icci
  solitons.
\newblock {\em Comm. Math. Helv.}, 77:297--325, 2002.

\bibitem[WZ04]{WZ}
Xu-Jia Wang and Xiaohua Zhu.
\newblock {K}ähler–{R}icci solitons on toric manifolds with positive first
  {C}hern class.
\newblock {\em Adv. in Math.}, 188:87 -- 103, 2004.

\bibitem[Yau78]{Y}
Shing-Tung Yau.
\newblock On the {R}icci curvature of a compact {K}ähler manifold and the
  complex {M}onge-{A}mp\`ere equation, {I}.
\newblock {\em Comm. on Pure and Applied Math.}, 31:339--411, 1978.

\bibitem[Zhu12]{ZZZ}
Xiaohua Zhu.
\newblock {K}\"ahler-{R}icci flow on a toric manifold with positive first
  {C}hern class.
\newblock In {\em Differential geometry}, volume~22 of {\em Adv. Lect. Math.
  (ALM)}, pages 323--336. Int. Press, 2012.

\end{thebibliography}
\end{document}